\documentclass[12pt, a4paper,reqno]{article}

\usepackage{palatino}
\usepackage{amsthm}
\usepackage{amsfonts}
\usepackage{amsmath}

\usepackage{epsfig}

\theoremstyle{plain}
\begingroup

\newtheorem{thm}{Theorem}[section]
\newtheorem{theorem}{Theorem}[section]

\newtheorem{lemma}[thm]{Lemma}

\newtheorem{proposition}[thm]{Proposition}

\endgroup

\theoremstyle{definition}
\newtheorem{defn}{Definition}[section]
\newtheorem{definition}{Definition}[section]

\newtheorem{remark}[defn]{Remark}

 \DeclareMathOperator{\im}{Im}

\DeclareMathOperator*{\res}{\mathrm{Res}}
\theoremstyle{remark}

\usepackage{helvet}
\usepackage{courier}
\usepackage{type1cm}

\usepackage{makeidx}         
\usepackage{graphicx}        
\usepackage{multicol}        
\usepackage[bottom]{footmisc}

\def\I{\mathrm{i}}

\def\D{{\mathbb D}}
\def\R{{\mathbb R}}
\def\C{{\mathbb C}}
\def\P{{\mathbb P}}
\def\U{{\mathbb U}}
\def\L{{\mathbb L}}
\def\Z{{\mathbb Z}}

\begin{document}

\title{Quadrature for quadrics}
\author{Bj\"orn Gustafsson\textsuperscript{1}}
\date{\today}
\maketitle

\abstract{We make a systematic investigation of quadrature properties for quadrics, namely integration of
holomorphic functions over planar domains bounded by second degree curves. A full understanding requires extending
traditional settings by allowing domains which are multi-sheeted, in other words domains which really are branched
covering surfaces of the Riemann sphere, and in addition usage of the spherical area measure instead of the Euclidean.

The first part of the paper discusses two different points of view of real algebraic curves: traditionally they live in the real projective plane,
which is non-orientable, but for their role for quadrature they have to be pushed to the Riemann sphere.

The main results include clarifying a previous theorem (joint work with V.~Tkachev), which says that a branched covering map 
produces a domain with the required quadrature properties if and only it extends to be meromorphic on the double of the 
parametrizing Riemann surface. In the second half of the paper domains bounded by ellipses, hyperbolas, parabolas and their inverses
are studied in detail, with emphasis on the hyperbola case, for which some of the results appear to be new.}



\noindent\footnotetext[1]
{Department of Mathematics, KTH Royal Institute of Technology, 100 44 Stockholm, Sweden.
Email: \tt{gbjorn@kth.se}}

\tableofcontents


\section{Introduction}

This note is inspired by a question by Henrik Shahgholian on why exterior of ellipses and parabolas in the plane are
null-quadrature domains, but not the exterior of hyperbolas.  Partial answers, from the point of view of potential theory
and PDE (partial differential equations), can be read off from papers of Shahgholian himself and other mathematicians, but here
we shall try to give a more philosophical answer in terms of what we call multi-sheeted algebraic domains, or quadrature Riemann surfaces.

The paper is more oriented towards real algebraic geometry than to PDE, but the two research areas are
linked via potential theory, specifically quadrature domains. The title of the paper is a little of a play with words: quadrature refers
to numerical integration via the latin noun `quadrature', making square-shaped. After having made a region square-shaped,
or dived it into small squares, one can read off the area of any arising square by taking the power two of (squaring) its side length. 
For related reasons, `quadrics' refers to curves described by equations of degree two. The word `conic' is more or less synonymous, 
but has a more geometric slant as the intersection between a cone and a plane. The actual subject of this paper eventually becomes integration
of holomorphic functions over domains bounded by algebraic curves of degree two, essentially ellipses, hyperbolas and parabolas.

While the subject of conic sections is very old,
potential theory dates back ``only'' to Newton's theory of gravitation (1687). But already Newton himself
realized that conic sections have a special role in gravitational theory. 
Besides the role of these curves in planetary motions there is the theorem
that ellipsoidal homoeoids  (elliptic homothetic shells)  give zero gravitational attraction inside the hole. This was proved in
three dimensions by Newton himself (at least to some part) and later extended by P.S~Laplace and J.~Ivory, 
see \cite{Kellogg-1967, Sakai-1983a}, and in particular the recent paper \cite{Izmestiev-Tabachnikov-2017} which contains
a wealth of geometric aspects. For the life of James Ivory, see \cite{Craik-2000}.
Several papers produced by the V.~Arnold school, for example
\cite{Arnold-1982, Arnold-1983, Vainstein-Shapiro-1985, Arnold-Khesin-1998}
contain further results for hyperboloids, and 
\cite{Sakai-1981, DiBenedetto-Friedman-1986, Friedman-Sakai-1986, Shapiro-1987, Entov-Etingof-1991, Shahgholian-1991a, Karp-1994, Karp-2008, 
Caffarelli-Karp-Shahgholian-2000, Khavinson-Lundberg-2014} give links between null quadrature domains and regularity theory for free boundaries. 

The property of ellipsoidal shells creating no interior gravity can be understood in terms of a Laplacian growth process (or Hele-Shaw
moving boundary problem) which preserves harmonic moments and which is related to quadrature domains. A few general sources are
\cite{Varchenko-Etingof-1992, Margulis-1995, Gustafsson-Teodorescu-Vasiliev-2014}.  See \cite{Gustafsson-Shapiro-2005} for a an
overview of the theory of quadrature domains.
Topology of quadrature domains, partly in connection with Hilbert's sixteenth problem, is studied in \cite{Lee-Makarov-2016}.  

In the present paper we make a systematic investigation of quadrature properties for planar domains bounded by second degree curves,
including to some extent inversions of such domains. We get rid of difficulties caused by exterior domains having infinite area by working with spherical
area measure instead of Euclidean area measure. This does not change the class of quadrature domains as a whole, in that sense the
change is inessential, however some computations become a more involved.

Another difference compared to traditional treatments is that we allow the domains to be ``multi-sheeted'', i.e. to be covered several times by a
branched covering map from some uniformizing surface. The latter is to be one
half of a symmetric, or ``real'', Riemann surface of ``dividing type''. This simply means the Riemann surface (assumed compact) has an 
anti-conformal involution and that it becomes disconnected after removal of the fixed points of the involution. 

The boundaries of so arising multi-sheeted quadrature domains, or quadrature Riemann surfaces in the terminology of Sakai \cite{Sakai-1988},
are algebraic curves, and essentially one half of all algebraic curves arise in this way. Therefore much of our treatment falls into the subject
of real algebraic geometry. But here arises a kind a dichotomy: real algebraic curves have their natural loci in the real projective plane $\R\P^2$
while the potential theoretic considerations rather take place in the Riemann sphere $\C\P^1\cong \C\cup\{\infty\}$.

Looking with positive eyes 
this actually makes the theory more interesting and rich, and there are even some surprises to come up. While the theories for the ellipse and
parabola are rather straight-forward, the hyperbola cannot be discussed without running into multi-sheeted domains and dichotomies between
$\R\P^2$ and $\C\P^1$: in $\R\P^2$ the hyperbola is an everywhere smooth curve (even at points of infinity), it is an ``oval'' in the terminology
of real algebraic geometry. The interior of the oval are the two pieces containing the focal points, while the exterior is the intermediate
region, which topologically is a M\"obius strip (recall that $\R\P^2$ is non-orientable). See Figure~\ref{fig:hyperbola2}. The corresponding
multi-sheeted region in $\C\P^1$ which has potential theoretic significance consists of one of the pieces containing the focal point, having multiplicity
two and with the focal points as branch point, continued with multiplicity one across the curve, up to the other component of the curve.
See Figure~\ref{fig:hyperbola3}. And in $\C\P^1$ the hyperbola has a singularity at infinity, in fact it becomes a lemniscate upon inversion. 

After an initial discussion of algebraic curves in projective spaces, partly in terms of corresponding projection maps (Sections~\ref{sec:generalities}
and \ref{sec:general}), we turn in Section~\ref{sec:quadrature} to basic definitions and general properties of quadrature domains in our setting. 
This part extends, and partly clarifies, a previous treatment \cite{Gustafsson-Tkachev-2011}. 
The single main result of the paper is Theorem~\ref{thm:qi}. Some computational details for its proof are deferred to an 
appendix Section~\ref{sec:appendix}. 

After these general matters in Sections~\ref{sec:generalities}--\ref{sec:quadrature} we discuss systematically in the remaining 
Sections~\ref{sec:ellipses}--\ref{sec:inversions}  
the different special cases arising for second degree curves: the ellipse, the hyperbola, the parabola, and inversions of these curves.
These sections are to a large extent computational in nature, but the computations are not always easy. One of the main outcomes
is that hyperbolas are indeed multi-sheeted quadrature domains. If one works with Euclidean area measure instead
of the spherical, then they are multi-sheeted null quadrature domains 
 
The text is accompanied with hopefully helpful pictures, constructed in most cases by ``Ti$k$Z''. 
These are however not computationally exact, they are mainly for illustrational purpose. 
 

\section{The genus and Riemann-Hurwitz formulas}\label{sec:generalities} 

For a plane algebraic curve $P(z,w)=0$ of degree $d$ and genus $\texttt{g}$ the {\it genus formula} says that
\begin{equation}\label{genus formula}
\texttt{g}+\text{number of singular points}= \frac{(d-1)(d-2)}{2}.
\end{equation}
It is assumed in this formula that the curve is irreducible, and for our considerations it will in addition be 
real in the sense that $P(z,w)$ is real-valued when $w=\bar{z}$.
The coordinate functions $z$ and $w$ can be viewed as
meromorphic functions on that compact Riemann surface $M$ which uniformizes the curve, and thereby 
these coordinate functions can also be viewed as branched covering maps of $M$ onto the Riemann sphere. 
For such a map, let $m$ denote the number of sheets and $\texttt{b}$ the number of branch points
(in general these will be different for $z$ and $w$). Then the {\it Riemann-Hurwitz formula} says that
\begin{equation}\label{Hurwitz}
2m-\texttt{b}=2(1-\texttt{g}).
\end{equation}
We refer to \cite{Kendig-1977, Griffiths-Harris-1978, Clemens-1980, Farkas-Kra-1992} for details. 

One fundamental difference between the genus formula
and the Riemann-Hurwitz formula is that the genus formula strictly speaking refers to the completion of the curve in two dimensional
complex projective space $\C\P^2$ (for example, one has to count the singular points within that space), 
and from that point of view the corresponding real algebraic curve $P(z,\bar{z})=0$  becomes a subset  of the real projective space $\R\P^2$.
On the other hand, (\ref{Hurwitz}) refers to a meromorphic function  $z :M\to\C\P^1$,
and this makes essentially the same curve become a subset of the Riemann sphere $\C\P^1$. The finite parts of $\R\P ^2$
and $\C\P^1$ can both be naturally be identified with $\R^2\cong \C$, but when it comes to points of infinity they differ considerably:
$\C\P^1$ has only one point of infinity whereas $\R\P^2$ has a whole projective line of infinities, representing asymptotic directions.
In addition, $\R\P^2$ is non-orientable. In fact, denoting by $S^2$ the unit sphere in $\R^3$, there are homoemorphisms $\C\P^1\cong S^2$ and 
 $\R\P ^2\cong S^2/(\text{antipodes identified})$.  

In the case of quadrics (conics) $d=2$, hence the genus formula says that $\texttt{g}=0$ and that there are no singular points. 
Thus $\texttt{b}=2(m-1)$, which gives $\texttt{b}=2$ in most cases to be discussed below. 
One exception is the circle, for which $m=1$ and so $\texttt{b}=0$.
We shall also discuss curves obtained by making inversions (such as the antipodal map $z\mapsto -1/\bar{z}$) of quadrics. This will lead to
curves like hippopedes (special cases of hypocycloids), lemniscates and cardioids.
Such inversions will not change the genus, since the curve is still uniformized by the same Riemann surface, but the degree of the curve may change.
Typically it will  change from $d=2$  to $d=4$, and then the genus formula tells that there will be three singular points. These will turn out
to have roles of being quadrature nodes, ``special points'', intersections between smooth branches, and cusps, depending on the curve.
In general, some of the singular points may be invisible in the real.


\begin{figure}
\begin{center}
\includegraphics[scale=0.8]{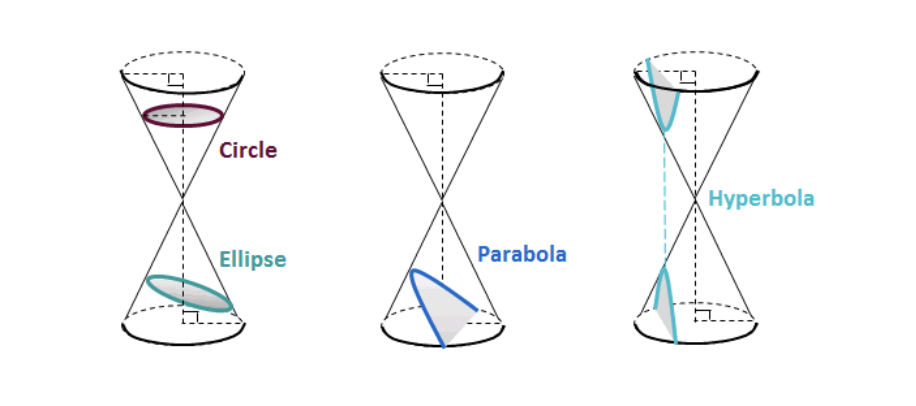}
\end{center}
\caption{The conic sections. Picture copied from mathhints.com.}
\label{fig:conics}
\end{figure}


\section{Algebraic curves in projective spaces}\label{sec:general}

\subsection{General}

We shall be concerned with real algebraic curves,
the meaning of ``real'' being that the polynomial $Q(x,y)=P(x+\I y, x-\I y)$ has real coefficients.
Then the coefficients of $P(z,w)$ has a hermitean symmetry,
and the complex curve  $\{(z,w)\in\C^2:P(z,w)=0\}$ admits the involution
$(z,w)\mapsto (\bar{w},\bar{z})$. In general we will denote all anti-conformal involutions by the same letter $J$,
the exact meaning then to be understood from the context. The real locus of $P$, namely 
\begin{equation}\label{V}
V=\{z\in \C: P(z,\bar{z})=0\},
\end{equation}
is the fixed point set of the mentioned involution.

The Riemann surface $M$ associated to the curve $P(z,w)=0$ has a corresponding anti-conformal  involution $J:M\to M$.
The locus (\ref{V}) can be  obtained from the fixed point set $\Gamma_M$ of $J$, which can also be viewed as the symmetry line of $M$,
essentially as the image $z(\Gamma_M)$ under the coordinate function $z$, the latter considered as a meromorphic function on $M$. 
It can alternatively be obtained using $w$, as $V=\overline{w(\Gamma_M)}$.

The above picture is actually somewhat simplified, for example one has to carefully compactify the above 
algebraic curves in projective spaces, and it also turns out that $V$ may contain accidental points, not
coming from $\Gamma_M$. We elaborate these details below.

The two dimensional real projective space is the set of equivalence classes (or ratios)
$$
\R\P^2=\{[\,t:x:y\,]: (t,x,y)\in\R^3\setminus (0,0,0) \},
$$
representing lines through the origin in $\R^3$. As indicated, only the ratio counts,
i.e.  $[\,t:x:y\,]=[\,\lambda t:\lambda x:\lambda y\,]$ for any $\lambda\ne 0$, $\lambda\in\R$.
The finite plane $\R^2$ can be naturally embedded in $\R\P^2$ via $(x,y)\mapsto [\,1:x:y\,]$.
The complex projective plane $\C\P^2$ is defined similarly, just that $t,x,y,\lambda$ are allowed to be complex
numbers. However, it will be more suitable in our context to name the coordinates $t,z,w$ and identify the real subspace
via $z=x+\I y$, $w=x-\I y$, where $x,y\in\R$. We shall thus work with
$$
\C\P^2=\{[\,t:z:w\,]: (t,z,w)\in\C^3\setminus \{(0,0,0)\} \},
$$
provided with the anticonformal  involution
$$
[\,t:z:w\,]\mapsto [\,\bar{t}:\bar{w}:\bar{z}\,].
$$
The set of fixed points for the involution is the ``real diagonal''  in $\C\P^2$, for which $t\in \R$ and $w=\bar{z}$.
Thus it consists of the points
$$
[\,t:z:\bar{z}\,]=[\,t:x+\I y:x-\I y\,] \quad (t,x,y\in\R).
$$

In this version of the real projective space the third component, $\bar{z}=x-\I y$, is redundant and can be discarded.
Therefore the real projective space can effectively be identified with the set of equivalence classes 
\begin{equation}\label{Delta}
\R\P^2_{\rm red}=\{[\,t:z\,]: \, t\in\R, z\in\C \}=\{[\,t:x+\I y\,]: \, t,x,y\in \R\},
\end{equation}
where two pairs represent the same point if they are related by a nonzero {real} factor. 
The subscript ${\rm red}$ stands for ``reduced'', meaning that an inessential component
has been discarded.
The definition of set $\R\P^2_{\rm red}$ reminds of that of the one dimensional complex projective space 
$$
\C\P^1=\{[\,t:z\,]:\, (t,z)\in\C^2\setminus (0,0)\},
$$
a model for the Riemann sphere. Here it is always possible to choose $t$ to be real, in fact the two values $t=0$ and $t=1$
cover all possible equivalence classes.  It follows that there is a natural and surjective map
\begin{equation}\label{DeltaCP}
\R\P^2_{\rm red}\to \C\P^1,
\end{equation}
which takes a point $[\,t:z\,]\in\R\P^2_{\rm red}$ to the point with the same name in $\C\P^1$. But in $\C\P^1$ space it may belong to a 
bigger equivalence class. This is more precisely the case when $t=0$:
in $\C\P^1$ all points with $t=0$ are equivalent to $[\,0:1\,]$, i.e. there is only one point of infinity, while
for $\R\P^2_{\rm red}$, two pairs $[\,0:z_1\,]$ and $[\,0:z_2\,]$ represent the same point of infinity
if and only if $z_2=\lambda z_1$ for some $\lambda \in\R\setminus \{0\}$. 
Thus one can take representatives of the form $[\,0:e^{\I\varphi}\,]$, observing then that two $\varphi$ represent the same point if 
and only if they differ by an integer multiple of $\pi$ (not $2\pi$). The latter remark is related
to the fact that $\R\P^2_{\rm red}$ is a non-orientable surface, topologically a ``cross-cap''.
The Riemann sphere, $\C\P^1\cong S^2$ is, on the other hand, orientable.  

In the sequel we shall also semi-complexify the real spaces by identifying $\R^3$ with $\R\oplus \C$ with 
coordinates $(t,z)=(t,x+\I y)$, and also identify $\R\P^2$ with $\R\P^2_{\rm red}$ as in (\ref{Delta}).


\subsection{Ray representations}

In general, projective spaces like $\R\P^2$, $\C\P^2$, $\R\P^1$ and $\C\P^1$ can be defined as the set of lines through
a fixed point in a vector space of one dimension higher than the index given. For example,
\begin{equation}\label{Klines}
\R\P^2=\{\text{lines }K\subset \R^3 \text{ through } (0,0,0)\},
\end{equation}
or, choosing another point, the north pole $N=(1,0,0)$,
\begin{equation}\label{Llines}
\R\P^2=\{\text{lines }L\subset \R^3 \text{ through }N\}.
\end{equation}

Each of the above two representations (\ref{Klines}) and (\ref{Llines}) has its own advantages. In the first case, each line $K$ intersects the unit sphere
$S^2$ at two antipodal points, $p$ and $-p$. Switching to $\R\P^2_{\rm red}$ 
this gives a bijective map
$$
\R\P^2_{\rm red}\to S^2/(\text{antipodes identified}). 
$$
\begin{equation}\label{KKp}
K\mapsto K\cap S^2=\{\pm p\}.\qquad\qquad
\end{equation}
Here one hemisphere can be discarded, say the ``southern'' hemisphere, which gives the description
$$
\R\P^2_{\rm red}\cong (\text{northern hemisphere}) \cup (\text{equator with antipodes identified}),
$$
presenting $\R\P^2_{\rm red}$ indeed as a ``cross-cap''. 

As a further concretization of (\ref{Klines}), each line $K$ which is not horizontal intersects the plane $\{t=1\}$
at exactly one point, $(1,z)\in\R\oplus \C$. This point corresponds to the projective coordinates for the line when written as 
$[\,1:z\,]$, hence corresponds directly to a complex number $z$, which we denote  $z_K=z_K(p)\in\C$, with $p$ as in (\ref{KKp}). 
Obviously $z_K(-p)=z_K(p)$.
Any horizontal line represents a points of infinity for $\R\P^2_{\rm red}$,
and such a line has a projective representation as $[\,0: e^{\I \varphi}\,]$ with $\varphi\in\R$ taken modulo $\pi$. 
See Figure~\ref{fig:stereo} for an illustration.

The second choice (\ref{Llines}) is connected to the stereographic projection, which rather identifies
$S^2$ with $\C\cup \{\infty\}$. Each non-horizontal line $L$ through the north pole
$N=(1,0)\in\R\oplus\C$ intersects $S^2$ in one more point, say $p$, and it also intersects the plane $\{t=0\}$ at one point, 
say $(0,z)$, where then $z=z_L(p)\in\C$. This defines the stereographic projection
$$
S^2\setminus \{N\}\to \C, \quad p\mapsto z_L(p),
$$
which extends to all of $S^2$ by setting $z_L(N)=\infty$.  Thus the horizontal lines $L$ give via the stereographic
projection only one point of infinity, despite there as as many horizontal lines of $L$ as there are of $K$.
On the computational side, the relation between the components of $p\in S^2\setminus\{N\}$
and $z=z_L(p)\in \C$ in the stereographic projection is
\begin{equation}\label{ps}
 p=(\,\frac{|z|^2-1}{|z|^2+1}, \frac{2z}{|z|^2+1})\in \R\oplus\C.
\end{equation}
In the other direction we have
\begin{equation}\label{zzL}
z=z_L(p)=\frac{x_1+\I x_2}{1-x_0}
\end{equation}
if we identify $p$ as the point $p=(x_0,x_1+\I x_2)\in S^2\subset \R\oplus \C$.

The Riemannian metric on $S^2$ as inherited from the ambient Euclidean space is
$$
ds^2=dx_0^2+dx_1^2+dx_2^2,
$$
and it is straightforward to show (and of course well-known) that this becomes
$$
ds^2=\frac{4|dz|^2}{(1+|z|^2)^2}
$$
when expressed in terms of $z=z_L(p)$. The fact that this is of the form $ds^2=\lambda(z)^2|dz|^2$
for some scalar factor $\lambda(z)>0$ means that the map $p\mapsto z_L(p)$ is conformal as a map
between $S^2$ and the Riemann sphere $\C\cup\{\infty\}$.

Thinking next of the Euclidean coordinates (\ref{ps}) for $p\in S^2$ as projective coordinates they are, in this role, 
$$
[\,\frac{|z|^2-1}{|z|^2+1}: \frac{2z}{|z|^2+1}\,]
=[\,|z|^2-1:2z\,]=[\,1:\frac{2z}{|z^2|-1}\,]\in\R\P^2_{\rm red}.
$$
This means that the complex number ${2z_L}/({|z_L|^2-1})$ represents the line $K$, provided $|z_L|\ne 1$. 
Accordingly we have, with $z_K=z_K(p)$ as above,
\begin{equation}\label{zKL}
z_K=\frac{2z_L}{|z_L|^2-1},
\end{equation}
or 
\begin{equation}\label{zKzL}
\frac{1}{z_K}=\frac{1}{2}(\bar{z}_L-\frac{1}{z_L}),
\end{equation}
the latter showing that $1/z_K$ is a harmonic function of $1/z_L$. Expressed in terms of $p=(x_0,x_1+\I x_2)$ we have
\begin{equation}\label{zK}
z_K(p)=\frac{x_1+\I x_2}{x_0}. 
\end{equation}
In contrast to $z_L(p)$, the map $p\mapsto z_K(p)$ is {not} conformal. Indeed, 
the relationship (\ref{zKL}) between $z_L$ and $z_K$ is not analytic (or anti-analytic),
hence $z_L$ and $z_K$ cannot both be conformal. 

When $|z_L|=1$ the line $K$ is horizontal, which means that any such $z_L$ represents a point of infinity in $\R\P^2$, then with $z_L$ 
and $-z_L$ representing the same point. We finally note (or recall) that
$$
z_L(-p)=-\frac{1}{\overline{z_L(p)}}, \quad z_K(-p)=z_K(p).
$$


\begin{figure}
\begin{center}
\includegraphics[width=\textwidth, scale=0.8, trim=100 350 5 200]{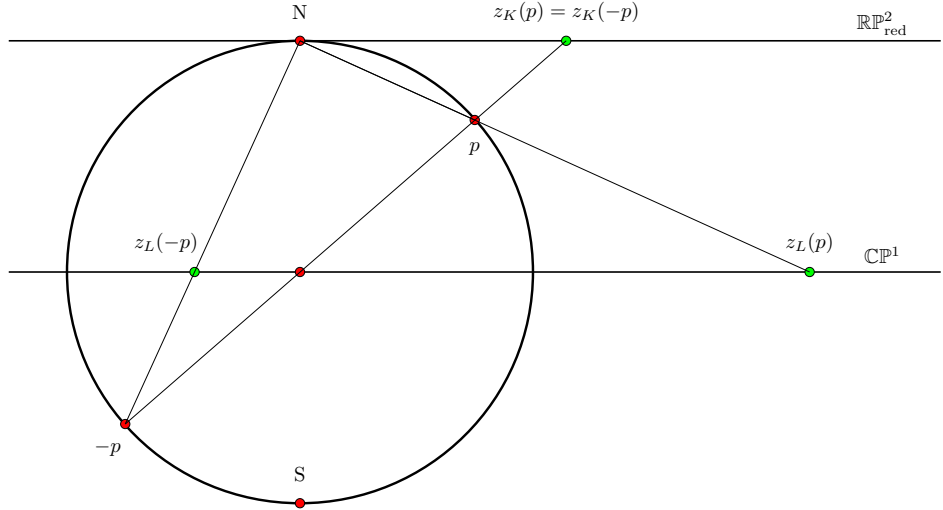}
\end{center}
\caption{Illustration of projections $p\mapsto z_L(p)$ and $p\mapsto z_K(p)$ of sphere onto complex
and real planes with their infinities (two-dimensional view).}
\label{fig:stereo}
\end{figure} 


We summarize everything as follows (see again Figure~\ref{fig:stereo}).
\begin{lemma}\label{lem:pp} 
Points $p=(x_0,x_1+\I x_2)$ on the unit sphere $S^2\subset \R\oplus\C$ have two natural representations as (extended)
complex numbers, namely as  $z_L$ and $z_K$ defined below.

1) Via the stereographic projection, which is a conformal map 
$$
z_L: \quad S^2\to \C\cup \{\infty\}\cong \C\P^1
$$
defined by
$$
z_L(p):=
\begin{cases}
\frac{x_1+ \I x_2}{1-x_0}, \quad &  p\ne N\quad  (x_0\ne 1),\\
\infty, \quad & p=N \quad (x_0=1).
\end{cases}
$$
Alternatively, considering $\C\P^1$ to be the target for stereographic projection,
$$
z_L(p):=
\begin{cases}
 [\,1-x_0 : x_1+\I x_2\,], \quad &  p\ne N,\\
[\,0:1\,], \quad & p=N.
\end{cases}
$$
In the converse direction $p$ can be recovered from $z_L=z_L(p)$ by
$$
p=(\frac{|z_L|^2-1}{|z_L|^2+1}, \frac{2z_L}{|z_L|^2+1}).
$$

2) As a map onto the real projective plane, bijective after identification of antipodes:
$$
S^2/J\to \R\P^2_{\rm red}.
$$
This is defined by
$$
 \pm p\mapsto
\begin{cases}
 [ \, 1:z_K(p)\, ],\quad &x_0\ne 0,\\
[\,0:e^{\I \varphi}\,], \quad &x_0=0,
\end{cases}
$$
where 
$$
z_K(p)= \frac{x_1+\I x_2}{x_0}, \quad x_0\ne 0,
$$
and $\varphi \in \R/\pi \Z$ is defined by $x_1+\I x_2=e^{\I\varphi}$ when $x_0=0$
\end{lemma}

The maps in the lemma are, in summary,
\begin{equation}\label{diagram}
\C\cup\{\infty\} \stackrel{\cong}{\longrightarrow}\C\P^1\stackrel{\cong}\longrightarrow S^2 \longrightarrow S^2/J 
\stackrel{\cong}\longrightarrow\R\P^2_{\rm red}\longrightarrow\C\P^1\stackrel{\cong}\longrightarrow \C\cup\{\infty\}
\end{equation}
For finite points the associations are given by
$$
z_L\mapsto [\,1: z_L\,]\mapsto p \mapsto \{\pm p\}\mapsto [\,1:x_K+\I y_K\,]\mapsto [1:z_K\,]\mapsto z_K
$$
and for points of infinity by
$$
\infty\mapsto [\,0:1\,]\mapsto N\mapsto \{N,S\} \mapsto [\,0:e^{\I\varphi}\,]\mapsto [\,0:1\,]\mapsto \infty
$$
where $S=(-1,0)$ is the south pole.


\section{Quadrature domains}\label{sec:quadrature}

\subsection{Classical quadrature domains}

The traditional definition of a {\it quadrature domain} \cite{Gustafsson-Shapiro-2005} is that it is a bounded domain (or open set) 
$\Omega\subset \C$ for which there exist finitely many points $z_1,\dots,z_m\in\Omega$, integers $m_j\geq 1$, 
and coefficients $a_{kj}\in\C$ such that the identity
\begin{equation}\label{qi0}
\frac{1}{\pi}\int_\Omega h(z) dxdy=\sum_{k=1}^m\sum_{j=0}^{m_k-1} a_{kj} h^{(j)}(z_k)
\end{equation}
holds for all integrable analytic functions $h$ in $\Omega$. The simplest example is the unit disk $\D$, for which
the mean-value property for analytic functions gives the quadrature identity
\begin{equation}\label{qidisk}
\frac{1}{\pi}\int_\D h(z)dxdy = h(0).
\end{equation}

Quadrature domains are known to be bounded by algebraic curves.
In fact, if $\Omega$ satisfies (\ref{qi0}) then $\partial\Omega$ is a real algebraic variety as in (\ref{V}),
possibly minus finitely many points, which then are singular point of the variety. Solving the algebraic equation
$P(z,\bar{z})=0$ in (\ref{V}) for $\bar{z}$ gives the {\it Schwarz function} $S(z)$ for $\partial\Omega$. Then
$P(z,S(z))=0$ holds identically, and
\begin{equation}\label{Sz}
S(z)=\bar{z} \,\, {\rm on }\,\,\partial\Omega.
\end{equation}
The latter equation is the defining property of a Schwarz function, generally required to be holomorphic in an at least a 
one-sided neighborhood of $\partial\Omega$. 
The main statement \cite{Aharonov-Shapiro-1976} in this context is that $\Omega$ is a quadrature domain if and only if 
$\partial\Omega$ has a Schwarz function which extends to be  meromorphic in all of $\Omega$. 

A few sources in book form for quadrature domains are \cite{Davis-1974, Sakai-1982, Shapiro-1992, Varchenko-Etingof-1992}.
In the latter text quadrature domains as in (\ref{qi0}) are named an {\it algebraic domains}. This is an appropriate terminology
even though quadrature domains only represent a small subclass of all domains bounded
by algebraic curves. In order to diminish the gap attempts were made in
\cite{Gustafsson-Tkachev-2011} to extend the class of quadrature domains/algebraic domains
 by first of all working with the spherical area measure on the Riemann sphere
instead of Euclidean measure in the plane, and secondly by allowing multiply covered domains. The so arising class of
{\it multi-sheeted algebraic domains}, or by another name, used in \cite{Sakai-1988},  {\it quadrature Riemann surfaces}, 
covers essentially one half of all domains bounded by algebraic curves. This class consists of those multiply covered
domains which are images of one side of a symmetric (or ``real'') Riemann surface of ``dividing type'' under a 
function which is meromorphic on the entire surface. A symmetric Riemann surface is said to be of dividing type if the symmetry line
disconnects the surface onto two halves.

One simple example of a domain bounded by an algebraic curve, which is not a quadrature domain in the sense of (\ref{qi0}), is the interior
of any non-circular ellipse, for example the domain defined by
$$
x^2+2y^2<1.
$$
On the other hand, the exterior (with respect to the complex plane) of the ellipse is an unbounded domain admitting an identity (\ref{qi0}), 
in fact it is a {\it null quadrature domain}:
\begin{equation}\label{nullqd}
\int_{x^2+2y^2>1} h\,dxdy=0
\end{equation}
holds for all integrable holomorphic functions $h$ in the domain. See \cite{Sakai-1981} for a complete result in this direction.
And in a general sense the ellipse can be rescued to belong to the class of multisheeted algebraic domains.

A domain which cannot be rescued by any means is 
\begin{equation}\label{hyperellipse}
x^4+y^4<1.
\end{equation}
The two Riemann surfaces associated to this domain, namely that defined by the algebraic equation
for the boundary and the Schottky double of the domain, are not canonically isomorphic in the way they need
in order to be a quadrature domain. As we shall try to argue below, almost everything goes wrong with the curve
in (\ref{hyperellipse}).

Generic algebraic curves are smooth, i.e. have no singular points. However algebraic curves bounding quadrature domains
as in (\ref{qi0}) usually have many singular points when completed in complex projective space. 
As follows from a classification carried out in \cite{Gustafsson-1988} the right member in
(\ref{qi0}) is a manifestation of some of these singular points. The functional in right member accounts more precisely for $n(n-1)$ of the singular
points, $n=\sum_{k=1}^m m_k$ being the {\it order} of the quadrature identity. The curve bounding the domain in (\ref{hyperellipse}) can be seen 
to have no singular points at all, even when viewed as an algebraic curve in $\C\P^2$. 
This fact is in itself a sign that it cannot bound any quadrature domain.

Using the genus formula (\ref{genus formula}), the absence of singular points 
in addition shows that the mentioned algebraic curve has genus three, while the real locus, defining the boundary in (\ref{hyperellipse}),
consists of only one closed curve. One single closed curve cannot separate a surface of genus three into two equal halves, hence the 
Riemann surface for (\ref{hyperellipse}) cannot be of ``dividing type''. Finally, the Schottky double of the domain in (\ref{hyperellipse}) 
obviously has genus zero, not three.

Somewhat similar results as those for ellipses hold for parabolas, but hyperbolas are different, and indeed quite interesting.
We start by considering just one example, namely the open set defined by
\begin{equation}\label{hyperbola00} 
x^2-y^2>1.
\end{equation}
As a subset of the complex plane, or of the Riemann sphere, it has two components. As seen from the point of infinity,
for example after an antipodal inversion, the two components look like the two holes in  the figure ``$\infty$'', i.e. the curve becomes a lemniscate. 
However, when viewed as a subset of $\R\P^2_{\rm red}$ the set (\ref{hyperbola00}) becomes connected via all asymptotic 
lines with directions $y=cx$, $|c|<1$. 
The curve $x^2-y^2=1$ is an ``oval'' in a terminology used in real algebraic geometry, see {\cite{Shafarevich-1977, Gross-Harris-1981, Harris-1992},
and the domain defined by (\ref{hyperbola00}) is in that terminology the interior of it.

The exterior of the oval, i.e.  the complement of (\ref{hyperbola00}) in $\R\P^2_{\rm red}$, 
is topologically a M\"obius strip. This is completely consistent with $\R\P^2_{\rm red}$ being topologically a ``cross-cap'', which is non-orientable
and can be obtained by sewing a disk and a M\"obius strip along their unique boundary curves. One can see the topology directly by providing
the two branches of the hyperbola with positive directions downwards (decreasing $y$).  This does not seem correct for a boundary of a domain,
but indeed it is in this case, as we try to clarify in Figure~\ref{fig:hyperbola2}. For example, the interior of the oval consists of two pieces in the
picture, but these come with opposite orientation because of the  identification of diametrically opposite points infinitely far away.


\begin{figure}
\begin{center}
\includegraphics[width=\textwidth, scale=1.2, trim=150 300 0 200]{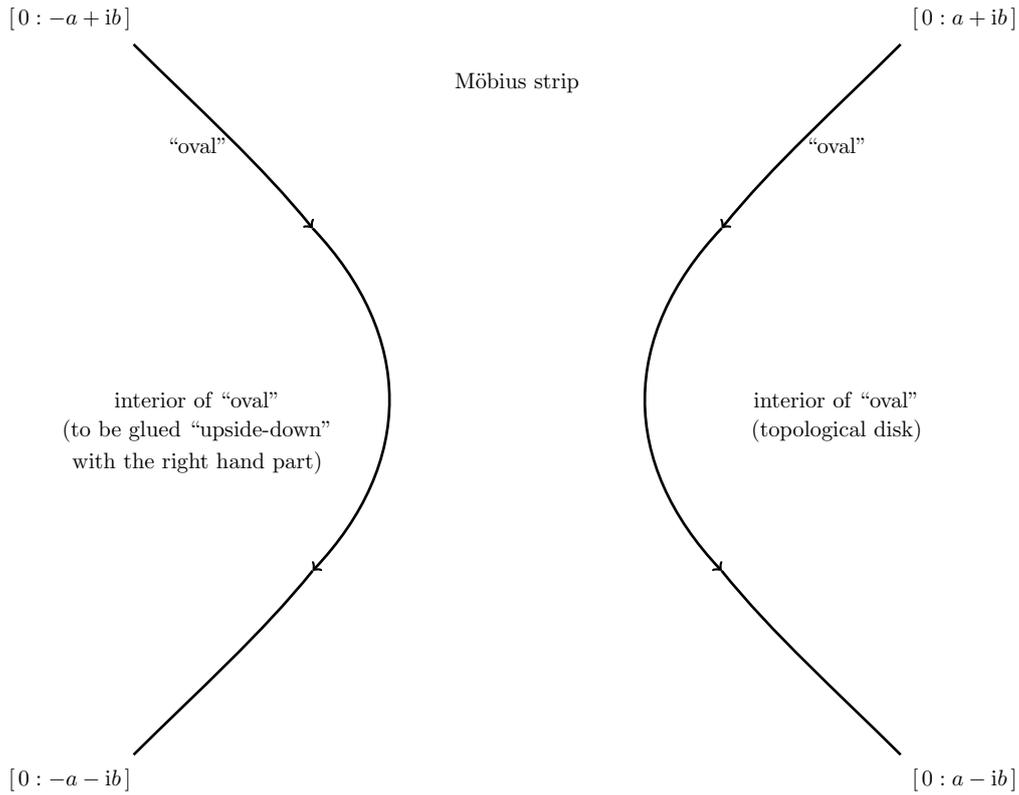}
\end{center}
\caption{The hyperbola $\frac{x^2}{a^2}-\frac{y^2}{b^2}=1$ as a non-singular curve in $\R\P^2_{\rm red}$,
passing through the two points of infinity $[\,0:-a-\I b\,]=[\,0:a+\I b\,]$ and $[\,0:-a+\I b\,]=[\,0:a-\I b\,]$.
The arrows indicate how the hyperbola becomes the single boundary component of a M\"obius strip.}
\label{fig:hyperbola2}
\end{figure} 


\begin{figure}
\begin{center}
\includegraphics[width=\textwidth, scale=1, trim=50 320 50 220]{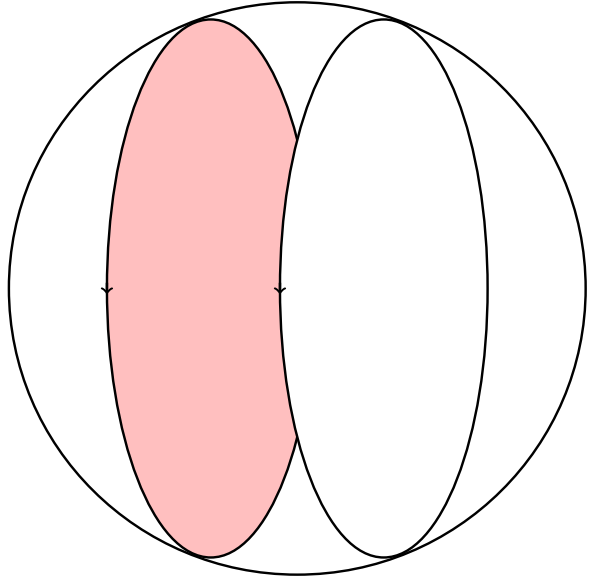}
\end{center}
\caption{Unusual picture of hyperbola: the boundary curves, as in  (\ref{xxx}), of two disks become a hyperbola when antipodal points 
on the sphere are identified. The sphere is in this role a double cover of the projective plane.}
\label{fig:hyperbola5}
\end{figure}


The hyperbola $x^2-y^2=1$ is also the conic obtained by intersecting of the cone 
$$
x^2-y^2=t^2
$$
with the plane $t=1$. Here the cone has two roles: besides being a choice of standard cone as in Figure~\ref{fig:conics},
it gives the equation for our special hyperbola expressed in homogenous coordinates. Thinking thus of $t,x,y$ as
homogenous coordinates and changing names to $x_0,x_1,x_1$, it is natural to seek a representative of
$[\,t:x+\I y\,]=[\,x_0:x_1+\I x_2\,]$ on $S^2$. Such a point $p=(x_0,x_1+\I x_2)$ is to satisfy $x_0^2+x_1^2+x_2^2=1$, in addition to
$x_1^2-x_2^2=x_0^2$. This gives the two circles
\begin{equation}\label{xxx}
x_1=\pm\frac{1}{\sqrt{2}}, \quad x_0^2+x_2^2=\frac{1}{2},
\end{equation}
representing the intersection of $S^2$ with two parallel planes. See Figure~\ref{fig:hyperbola5}. 

The interpretation of the above is simply that we have identified, via (\ref{xxx}), the tracks $\pm p\in S^2$ of the hyperbola
$x^2-y^2=1$ under the identification $z_K(p)=x+\I y$ in Lemma~\ref{lem:pp}. Note that the two circles become only
one circle in $S^2/J$ and that the intermediate part, defined by $-1/\sqrt{2}<x_1<1/\sqrt{2}$, becomes a M\"obius strip. 


\subsection{Quadrature domains with respect to spherical measure}

Leaving now the strict realm of (\ref{qi0}) and aiming at being able to treat domains with infinite area 
without loss of test functions, we shall replace ordinary area measure with spherical
area measure. We identify this with the two-form
\begin{equation}\label{dA}
\frac{dx \wedge dy}{\pi (1+  (x^2+y^2))^2}=\frac{d\bar{z}\wedge dz}{2\pi\I (1+  z\bar{z})^2},
\end{equation}
here scaled so that the entire space $\C\cup\{\infty\}$, representing the sphere $S^2$, has area one.

If $h(z)$ is analytic in a bounded domain $\Omega\subset \C$ and if we set $S(z)=\bar{z}$ 
on $\partial \Omega$, as in (\ref{Sz}), then partial integration gives
$$
\int_\Omega \frac{h(z)d\bar{z}\wedge dz}{(1+ z\bar{z})^2}
=\int_{\partial\Omega}\frac{h(z)\bar{z} dz}{1+  z\bar{z}}=\int_{\partial\Omega}\frac{h(z)S(z) dz}{1+  zS(z)}.
$$
From this identity it follows that $\Omega$ is a quadrature domain for the spherical measure if and only if 
the differential $\frac{S(z)dz}{1+zS(z)}$ has a meromorphic continuation to all of $\Omega$. 
This holds if and only if $S(z)$ is itself meromorphic in $\Omega$, hence $\Omega$ is a spherical quadrature domain if and only
if it is a Euclidean quadrature domain. The spherical quadrature identity then becomes
\begin{equation}\label{intres}
\frac{1}{2\pi\I}\int_\Omega \frac{h(z)d\bar{z}\wedge dz}{ (1+  |z|^2)^2}= \sum_{z\in\Omega}\,{\rm Res\,}\frac{h(z)S({z}) dz}{1+  zS(z)},
\end{equation}
where the right-hand side is a functional, as acting on $h$, on the same form as the right member in (\ref{qi0}).

A different usage of partial integration gives 
$$
\frac{1}{2\pi\I}\int_\Omega \frac{h(z)d\bar{z}\wedge dz}{(1+ z\bar{z})^2}
=\lim_{\varepsilon\to 0}\frac{1}{2\pi\I} \int_{\Omega\setminus \D(0,\varepsilon)} \frac{h(z)d\bar{z}\wedge dz}{(1+ z\bar{z})^2}
$$
$$
=\lim_{\varepsilon\to 0}\frac{1}{2\pi\I}\int_{\partial\D(0,\varepsilon)}\frac{h(z)dz}{z(1+  z\bar{z})}-\frac{1}{2\pi\I}\int_{\partial\Omega}\frac{h(z)dz}{z(1+  z\bar{z})}
$$
$$
=\lim_{\varepsilon\to 0}\frac{1}{2\pi\I}\int_{\partial\D(0,\varepsilon)}\frac{h(z)dz}{z}-\frac{1}{2\pi\I}\int_{\partial\Omega}\frac{h(z)dz}{z(1+  zS(z))}.
$$
Here one sees that if $S(z)$ has no pole at $z=0$, then the two terms in the final member both contribute with residues at $z=0$, but that these 
contributions cancel each other. Thus one can, if $S(z)$ is regular at $z=0$, write the final quadrature formula as 
\begin{equation}\label{intres1}
\frac{1}{2\pi\I}\int_\Omega \frac{h(z)d\bar{z}\wedge dz}{ (1+  |z|^2)^2}
= -\sum_{z\in\Omega\setminus \{0\}}\,{\rm Res}\Big(\frac{h(z)}{1+  zS(z)}\cdot\frac{dz}{z}\Big).
\end{equation}

In order to discuss multiply covered domains the coordinate $z$ has to be substituted by a function $z=f(\zeta)$, where $\zeta$ is a 
coordinate on a uniformizing  Riemann surface $M$, assumed to be ``real'' and of ``dividing type''. 
These assumptions on $M$ mean more precisely that $M$ shall be compact and be provided with an anti-conformal involution 
$J:M\to M$ such that, denoting by $\Gamma_M$ the set of fixed points of $J$,
$M\setminus \Gamma_M$ has exactly two components, to be denoted $M_+$ and $M_-$. It is assumed that $M$ itself is connected.
One of the components, say $M_+$, is selected to be a model for the multi-sheeted domain $\Omega$.

With $M$ as above, any non-constant meromorphic function $f$ on $M$ gives rise to a spherical quadrature domain.
Such a function may be viewed as a holomorphic and possibly branched covering map  $M\to \C\P^1$, and
the resulting quadrature domain $\Omega$ is to be identified with $f(M_+)$ with appropriate multiplicities. 
The multivalued analytic functions on $\Omega$ allowed as test functions
in quadrature identities will be those which become single-valued when lifted to $M_+$. 
Thus the test functions will effectively be integrable analytic functions on $M_+$ itself.

We shall sometimes refer to $M$ as the `parameter' plane (or $\zeta$-plane) and the target space $\C\P^1$
for $f$ as the `physical' plane' (or $z$-plane). 
In terms of the above notations we formalize the concept of a multi-sheeted algebraic domain
exactly as in \cite{Gustafsson-Tkachev-2011}: 

\begin{definition}\label{def:multi-sheeted algebraic domain}
A {\it multi-sheeted algebraic domain} is a pair $(M_+,f)$, where $M=M_+\cup\Gamma_M\cup M_-$ 
is a real compact Riemann surface of dividing type and $f$ is a non-constant meromorphic function on $M$.
Two such pairs, $(M_+,f)$ and $(\tilde{M}_+,\tilde{f})$, are to be considered same if there is a biholomorphic mapping
$\phi:M\to\tilde{M}$ such that $\phi\circ J=\tilde{J}\circ \phi$ and $f=\tilde{f}\circ \phi$, where $J$ and $\tilde{J}$ denote 
the involutions on $M$ and $\tilde{M}$, respectively.
The above means that it is only the image $\Omega=f(M_+)$ with appropriate multiplicities that counts. 
\end{definition}

It was shown in \cite{Gustafsson-Tkachev-2011} that starting with $M$ as in Definition~\ref{def:multi-sheeted algebraic domain} and $f$
a non-constant meromorphic function initially defined only on $M_+$, then $f$ extends to be meromorphic on all of $M$ if and only
if a quadrature identity
\begin{equation}\label{qisf}
\frac{1}{2\pi \I} \int_{M_+} h\,\frac{d\bar{f}\wedge d{f}}{(1+|f|^2)^2} 
=\sum_{k=1}^m\sum_{j=0}^{m_k-1} \alpha_{kj} h^{(j)}(\zeta_k)
\end{equation}
holds for all functions $h$ holomorphic and integrable with respect to the indicated spherical measure in $M_+$.
Here $\alpha_{kj}\in\C$, $\zeta_k\in M_+$ are fixed quadrature data, and the derivatives are to be taken with respect to 
suitable local coordinates around the quadrature nodes. One can also express the quadrature property (\ref{qisf}) without specifying the
coefficients $\alpha_{kj}$ by saying that the left member shall vanish for all holomorphic functions $h$ which vanish at the points
$\zeta_k$ to order at least $m_k$ ($1\leq k\leq m$). 

The ``easy'' implication in the above statement, saying that $f$ being meromorphic on all
of $M$ implies a formula of type (\ref{qisf}), is of course mainly a matter
of computation (partial integration combined with a residue calculus). However, it actually is slightly more delicate than one may at first expect. 
We shall therefore rework this computation, which was not fully carried out in \cite{Gustafsson-Tkachev-2011}.


\subsection{Residue considerations}\label{sec:residues}

A first remark is that one may avoid certain troubles by rotating the Riemann sphere.
M\"obius transformations which preserve the spherical metric are of the form
\begin{equation}\label{Mobius}
w=\frac{az+b}{-\bar{b}z+\bar{a}}, \quad |a|^2+|b|^2=1.
\end{equation}
These orientation preserving rigid transformations can be combined with complex conjugations so that inversions like $w={1}/{\bar{z}}$
and antipodal map $w=-{1}/{\bar{z}}$ are included. Post-composing meromorphic maps like $f$ above with
such M\"obius transformation allows us to rotate the Riemann sphere into suitable positions. 

In particular one can arrange that $f$ has no poles on the symmetry line $\Gamma_M$. 
Indeed, if $f$ has such a pole one may consider instead any function
\begin{equation}\label{g}
g(\zeta)=\frac{af(\zeta)+b}{-\bar{b}f(\zeta)+\bar{a}}, \quad |a|^2+|b|^2=1,
\end{equation}
without such poles. Then
$$
\frac{d\bar{f}\wedge df}{(1+f\bar{f})^2}=\frac{d\bar{g}\wedge dg}{(1+g\bar{g})^2},
$$
and one may work with $g$ in place of $f$.
The above also means that the spherical area form itself is not sensitive for poles, or any other specific values, of $f$.

Since the spherical area form, being a $2$-form, is automatically closed, it is locally exact.
For example, away from poles of $f$ we can write
\begin{equation}\label{nopoles}
\frac{d\bar{f}\wedge df}{(1+f\bar{f})^2}=d\Big(\frac{\bar{f}df}{1+f\bar{f}}\Big).
\end{equation}
Similarly, away from zeros of $f$,
\begin{equation}\label{nozeros}
\frac{d\bar{f}\wedge df}{(1+f\bar{f})^2}=-d\Big(\frac{df}{f(1+f\bar{f})}\Big).
\end{equation}
However the spherical area form cannot be globally exact because
\begin{equation}\label{intm}
\int_M\frac{d\bar{f}\wedge df}{(1+f\bar{f})^2}= m\ne 0.
\end{equation}
Here $m$ is the degree of $f$ (the number of zeros or poles), and we have used that the unit sphere has total area one.

Starting from the right member of (\ref{nozeros}) and computing the exterior derivative there, noting that $1/(1+f\bar{f})$
is always a smooth function and taking, in remaining factors, distributional contributions carefully into account and using
that $df\wedge df=0$, gives
$$
-d\Big(\frac{1}{1+f\bar{f}}\cdot \frac{df}{f}\Big)=\frac{d(f\bar{f})}{(1+f\bar{f})^2}\wedge \frac{df}{f} -\frac{1}{1+f\bar{f}}\cdot d\Big(\frac{df}{f}\Big)
$$
$$
=\frac{d\bar{f}\wedge df}{(1+f\bar{f})^2} -\frac{1}{1+f\bar{f}}\cdot d\Big(\frac{df}{f}\Big).
$$
Thus, rearranging terms,
\begin{equation}\label{zerosallowed1}
\frac{d\bar{f}\wedge d{f}}{(1+f\bar{f})^2}=-d\Big(\frac{df}{f(1+f\bar{f})}\Big)+\frac{1}{1+f\bar{f}}\cdot d\big(\frac{df}{f}\big).
\end{equation}
In a similar manner,
\begin{equation}\label{polesallowed1}
\frac{d\bar{f}\wedge d{f}}{(1+f\bar{f})^2}=d\Big(\frac{\bar{f}df}{1+f\bar{f}}\Big)-\frac{f\bar{f}}{1+f\bar{f}}\cdot d\big(\frac{df}{f}\big).
\end{equation}

In the above expressions the final terms consist of pure distributional contributions (Dirac currents):
\begin{equation}\label{ddff}
d\big(\frac{df}{f}\big)=\pi  \big(\sum_{k=1}^m \delta_{\omega_k}-\sum_{j=1}^m \delta_{\xi_j}\big)d\bar{\zeta}\wedge d\zeta,
\end{equation}
The right member here is the $2$-form current corresponding to the divisor of $f$,  $\omega_k$ denoting the zeros and $\xi_j$
the poles of $f$, both repeated according to multiplicities. 
Localized at a zero of $f$, for example, the identity expresses that $1/\pi z$ is a fundamental solution for the Cauchy-Riemann operator:
$$
\frac{\partial}{\partial \bar{z}} \big(\frac{1}{\pi z}\big)= \delta_0.
$$
Written in an invariant way the latter relation becomes
$$
d\big(\frac{dz}{z}\big)=\pi  \,\delta_0 \,d\bar{z}\wedge dz=2\pi\I \,\delta_0 \,dx\wedge dy,
$$
indicating also that Dirac measures shall be considered as $2$-form currents.
Compare Lemma~1 in \cite{Gustafsson-Tkachev-2009}.

Inserting (\ref{ddff}) into (\ref{zerosallowed1}) and (\ref{polesallowed1}), taking into account that the factor in front of 
$d\big(df/f\big)$ kills in each case half of the Dirac currents, gives the following formulas.

\begin{proposition}\label{prop:dfdf} 
Let $f$ be a meromorphic function on $M$ of degree $m$, 
let $\omega_k$ be its zeros and $\xi_j$ its poles, both repeated according to multiplicities.
Then, in the sense of currents,
\begin{equation}\label{addingzeros}
\frac{d\bar{f}\wedge df}{(1+f\bar{f})^2}
=-d\Big(\frac{df}{f(1+f\bar{f})}\Big)
+\pi \sum_{k=1}^m \delta_{\omega_k}d\bar{\zeta}\wedge d\zeta
\end{equation}
\begin{equation}\label{addingpoles}
\qquad\,=d\Big(\frac{\bar{f}df}{1+f\bar{f}}\Big)
+\pi \sum_{j=1}^m \delta_{\xi_j}d\bar{\zeta}\wedge d\zeta.
\end{equation}
\end{proposition}

\begin{remark}
The presence of the last terms in (\ref{addingzeros}) and (\ref{addingpoles}) was slightly slipped over in \cite{Gustafsson-Tkachev-2011}.
One the other hand, those same terms but with the opposite sign, usually appear also in the first terms 
and therefore eventually cancel. See Theorem~\ref{thm:qi} below.
\end{remark}

\begin{remark}\label{rem:zeros-poles}
The zeros and poles of $f$ are not really important in themselves, they can in the distributional contribution in 
(\ref{addingzeros}), (\ref{addingpoles}) be replaced by any other antipodal pair of values for $f$. 
For example, for any $c\in\C\P^1$ the poles can be replaced by those points $\eta_j$  for which 
$f(\eta_j)=c$. Writing $c=[\,\bar{b}:\bar{a}\,]$ in projective coordinates, normalized as in (\ref{g}),
the $\eta_j$ are the poles of the function $g$ in (\ref{g}), and in terms of $g$ we have
$$
\frac{d\bar{f}\wedge df}{(1+f\bar{f})^2}=d\Big(\frac{\bar{g}dg}{1+g\bar{g}}\Big)
+\pi \sum_{j=1}^m \delta_{\eta_j}d\bar{\zeta}\wedge d\zeta.
$$  
\end{remark}

\begin{remark}
Instead of treating poles within the framework of distributions or currents one may
cut small holes around singularities and then get additional boundary integrals, like in the derivation of (\ref{intres1}).
\end{remark}

Integrating over $M_+$ in Proposition~\ref{prop:dfdf}, using Stokes' theorem together with $\bar{f}=f^*$ on $\partial M_+$, gives the 
following versions of the quadrature identity (\ref{qisf}).

\begin{theorem}\label{thm:qi}
With notations and assumptions as in Proposition~\ref{prop:dfdf} we have the following quadrature identity for functions $h$
which are holomorphic in $M_+$ and integrable with respect to the pulled back, by $f$, spherical measure on 
the Riemann sphere:
\begin{equation}\label{sfqi1}
\frac{1}{2\pi \I} \int_{M_+} h\,\frac{d\bar{f}\wedge d{f}}{(1+|f|^2)^2} 
=-\sum_{M_+} {\rm Res\,} \frac{h  df}{f(1+ff^*)} +\sum_{\omega_k\in M_+} h(\omega_k)\qquad
\end{equation}
\begin{equation}\label{sfqi11}
\qquad\qquad
=+\sum_{ M_+} {\rm Res\,} \frac{h f^* df}{1+ff^*} +\sum_{\xi_j\in M_+} h(\xi_j).
\end{equation}
It is assumed, in accordance with Remark~\ref{rem:zeros-poles}, that
the sphere has been rotated so that zeros and poles of $f$ on $\partial M_+$ are avoided.

The first sums in the right members above are taken over all singular points in $M_+$ of the differentials indicated.
The second sums (with explicit zeros and poles) are generically contained, but with opposite signs, also in the first
sums and therefore usually cancel with terms from these. 

The last statement can be made precise 
under the assumption that $f^*$ does not vanish at the poles of $f$ and vice versa. In that case
it is enough to sum over those points $\zeta\in M_+$ for which $1+f(\zeta)f^*(\zeta)=0$, and the complete identities then become
\begin{equation}\label{sfqi2} 
\frac{1}{2\pi \I} \int_{M_+} h\,\frac{d\bar{f}\wedge d{f}}{(1+|f|^2)^2} 
=-\sum_{1+ff^*=0} {\rm Res\,} \Big(\frac{h  }{1+ff^*}\cdot\frac{df}{f}\Big)\qquad
\end{equation}
\begin{equation}\label{sfqi3}
\qquad\qquad\qquad\qquad\,\,  =+\sum_{1+ff^*=0} {\rm Res\,} \Big(\frac{h  }{1+ff^*}\cdot{f^*df}\Big).
\end{equation}
\end{theorem}

\begin{proof}
The formulas (\ref{sfqi1}), (\ref{sfqi11}) follow immediately on using Stokes' theorem together with 
Proposition~\ref{prop:dfdf}. Note that the factor $h$ can be multiplied
to (\ref{addingpoles}) and then goes inside the differential in the first term of the right member. 
For convenience we repeat the computation for (\ref{sfqi1}) in one sequence of equalities:
$$
\int_{M_+}h\frac{d\bar{f}\wedge d{f}}{(1+f\bar{f})^2}
=-\int_{M_+}d\Big(\frac{h\,df}{f(1+f\bar{f})}\Big)+\int_{M_+}\frac{h}{1+f\bar{f}}\cdot d\big(\frac{df}{f}\big)
$$
$$
=-\int_{\partial M_+}\frac{h}{1+f\bar{f}}\cdot\frac{df}{f}+\int_{M_+}\frac{h}{1+f\bar{f}}\cdot d\big(\frac{df}{f}\big)
$$
$$
=-\int_{\partial M_+}\frac{h}{1+f{f^*}}\cdot\frac{df}{f}+\int_{M_+}\frac{h}{1+f\bar{f}}\cdot d\big(\frac{df}{f}\big)
$$
$$
=-2\pi \I \,\sum_{M_+}{\rm Res\,}\Big(\frac{h}{1+f{f^*}}\frac{df}{f}\Big)+\int_{M_+}\frac{h}{1+f\bar{f}}\cdot d\big(\frac{df}{f}\big)
$$
$$
=-2\pi \I \sum_{M_+} {\rm Res\,} \frac{h  df}{f(1+ff^*)} +2\pi \I\sum_{\omega_k\in M_+} h(\omega_k).
$$
The use of Stokes' theorem can be  justified by using Ahlfors-Bers mollifiers \cite{Bers-1964} in a way which was introduced 
in \cite{Sakai-1982} and which is now a standard tool in the theory of quadrature domains. 

The  proof of (\ref{sfqi2}), (\ref{sfqi3}) under the stated assumption, i.e. confirmation of the presence of cancellations, 
consists of a lengthy but straight-forward local analysis of the differentials
\begin{equation}\label{differentials}
\frac{hdf}{f(1+ff^*)}, \quad \frac{hf^*df}{1+ff^*}
\end{equation}
near an arbitrary point. In order not to interrupt the main lines of presentation we defer these details to an appendix,
see Section~\ref{sec:appendix}.

\end{proof}

\begin{remark}\label{rem:varepsilon}
One may scale the spherical area measure with a parameter $\varepsilon>0$ to approach Euclidean area measure in the limit 
$\varepsilon\to 0$. Using that
$$
d\Big(\frac{ f\bar{f}}{1+\varepsilon^2  f\bar{f}}\cdot\frac{df}{f}\Big)
=\frac{d\bar{f}\wedge df}{(1+\varepsilon^2 f\bar{f})^2}+\frac{ f\bar{f}}{1+\varepsilon^2 f\bar{f}}\cdot d\Big(\frac{df}{f}\Big),
$$
which is a generalization of (\ref{polesallowed1}), one obtains the quadrature identity
$$
\frac{1}{2\pi \I} \int_{M_+} h\,\frac{d\bar{f}\wedge d{f}}{(1+\varepsilon^2 |f|^2)^2} 
=\sum_{M_+} {\rm Res\,} \frac{hf^*df  }{1+\varepsilon^2 ff^*}+\sum_{\zeta_j\in M_+}\frac{1}{\varepsilon^2}h(\xi_j).
$$
Like in the statements of the theorem, the contributions from the poles $\xi_j$ cancel under natural assumptions.
In the limit $\varepsilon\to 0$ the identity takes the traditional form
$$
\frac{1}{2\pi \I} \int_{M_+} h\,{d\bar{f}\wedge d{f}}
=\sum_{M_+} {\rm Res\,} {hf^*df  }.
$$

There may still be contributions from poles of $f$, but these come from solutions $\zeta\in M_+$ of $1+\varepsilon^2 f(\zeta)f^*(\zeta)=0$
which have become poles of $f$ in the limit $\varepsilon\to 0$. On the other hand, poles of $f$ make fewer functions $h$ be allowed in the left member.
For example, a pole of order one for $f$ requires $h$ to have a zero of order two at the same point (unless compensation comes from $f^*$),
and therefore contributions from poles of $f$ will usually not show up in the formula. This is what happens for null quadrature domains.
\end{remark}


\subsection{A simple example}

Just to illustrate the general concepts we take a very simple example.
Let $M=\C\cup\{\infty\}$, $J(\zeta)= \bar{\zeta}$,  $M_+=\{\zeta\in\C: \im \zeta >0\}$, $f(\zeta)=\zeta^n$
with $n>0$ an integer. Then $f^*(\zeta)=\zeta^n$ and
using (\ref{sfqi2}) we obtain, for $h$ analytic and integrable in $M_+$,
$$
\frac{1}{2\pi \I}\int_{M_+}h\,\frac{d\bar{f}\wedge df}{(1+|f|^2)^2}=-\sum_{M_+}\,\res_{1+ff^*=0}\,\Big(\frac{h}{1+ff^*}\cdot \frac{df}{f}\Big)
$$
$$
=-\sum_{\im\zeta>0}\res_{\zeta=\sqrt[2n]{-1}}\, \Big(\frac{ h(\zeta)}{1+\zeta^{2n}}\cdot \frac{n d\zeta}{\zeta}\Big)
$$
$$
=\frac{1}{2}\big(h(e^{\pi\I/2n})+h(e^{3\pi\I/2n})+\dots+h(e^{{(2n-1)}\pi\I/2n})\big).
$$

When $n=1$ or $n=2$, $f$ is univalent on $M_+$ and  the result can be written directly in the image domains $\Omega_1=\{\im z>0\}$ 
and $\Omega_2=\C\setminus [0,+\infty)$ as follows:
\begin{equation}\label{simple1}
\frac{1}{2\pi \I}\int_{\Omega_1}h(z)\,\frac{d\bar{z}\wedge dz}{(1+|z|^2)^2}=\frac{1}{2}h(\I),
\end{equation}
$$
\frac{1}{2\pi \I}\int_{\Omega_2}h(z)\,\frac{d\bar{z}\wedge dz}{(1+|z|^2)^2}=\frac{1}{2}\big(h(\I)+h(-(\I)\big).
$$
Here $h(z)$ is considered as a function of $z=f(\zeta)$. In the second formula, $\Omega_2$
can be replaced by $\C\setminus I$ where $I$ is an arbitrary closed subset of $\R$, and it then
actually follows from the first formula together with the corresponding formula for the reflected domain
$J(\Omega_1)$.

One may notice that $\partial\Omega_2$, and $\partial(\C\setminus I)$ in general, is not a full algebraic curve.
As remarked in \cite{Gustafsson-1983} (Lemma~1 and discussions after it) this is an exceptional situation showing up
in the present case because $f^*=f$, having as a consequence that the pair $f$ and $f^*$ does not generate the field of 
meromorphic functions on $M$
when $n\geq 2$.

With Euclidean area measure, in place of spherical, both $\Omega_1$ and $\Omega_2$ become null quadrature domains.


\section{The ellipse}\label{sec:ellipses}

\subsection{The disk}\label{sec:disk}

The most basic quadrature domain, the unit disk $\D=\{|z|<1\}$, is bounded by the real algebraic curve
\begin{equation}\label{circle}
\Gamma: \quad x^2+y^2=1.
\end{equation}
As a complex algebraic curve in $\C\P^2$ it can be identified with the Riemann sphere.

The Schottky double of $\D$ is also a model of the Riemann sphere, thought of
as $\C\cup\{\infty\}$ with anti-conformal involution $J(z)=1/\bar{z}$, which has $\Gamma$ as its fixed point set. 
The exterior of $\Gamma$ is (as for the conformal structure) just another copy of $\D$, and can be identified 
with the backside in the Schottky double. The fact that the two mentioned spheres are canonically identical, as Riemann surfaces with 
an anti-conformal involution, is crucial for the property of $\D$ being a quadrature domain. This is a general statement
applicable to all classical quadrature domains as in (\ref{qi0}). The quadrature identities for the disk and its exterior $\D^e$ are immediate: 
$$
\frac{1}{2\pi \I}\int_{\D}h(z)\,\frac{d\bar{z}\wedge dz}{(1+|z|^2)^2}=\frac{1}{2}h(0),
$$
$$
\frac{1}{2\pi \I}\int_{\D^e}h(z)\,\frac{d\bar{z}\wedge dz}{(1+|z|^2)^2}=\frac{1}{2}h(\infty).
$$
These identities can be identified with (\ref{simple1}) after suitable rotations of the sphere, as in (\ref{Mobius}).

The curve $\Gamma$ can be naturally viewed also as a curve in the real projective space $\R\P^2_{\rm red}$, which topologically
can be identified with the sphere $S^2$ with antipodes identified. What is outside the curve
then becomes topologically a M\"obius strip, hence a non-orientable surface. This is the generic situation in real algebraic geometry:
a simple closed algebraic curve, an ``oval'', which divides the space into a topological disk and a M\"obius strip. Hilbert's sixteenth problem
concerns the possible scenarios as for ovals inside each other in the case of real algebraic curves having several components, 
see {\cite{Shafarevich-1977, Wilson-1978, Gross-Harris-1981, Harris-1992}.


\subsection{The true ellipse}\label{sec:ellipse}

The standard ellipse in the plane is, with $a>b>0$,
\begin{equation}\label{ellipse}
\Gamma:\quad \frac{x^2}{a^2}+\frac{y^2}{b^2}=1,
\end{equation}
and it bounds the domain 
$$
E=\{(x,y)\in \R^2: \frac{x^2}{a^2}+\frac{y^2}{b^2}<1\}.
$$
In terms of the complex coordinate the equation (\ref{ellipse}) becomes
\begin{equation}\label{ellipse1}
z^2+\bar{z}^2-\frac{2(a^2+b^2)}{a^2-b^2}z\bar{z}+\frac{4a^2b^2}{a^2-b^2}=0,
\end{equation}
and solving for $\bar{z}$ gives a multivalued function with two branches as
$$
S_\pm (z)=\frac{a^2+b^2}{c^2}\,z \pm \frac{2ab}{c^2}\sqrt{z^2-c^2}.
$$
Here $\pm c$ are the foci, with
$$
c=\sqrt{a^2-b^2}>0,
$$

The branch $S_-(z)$ corresponding to the minus sign (taking the square root to be positive for large positive
values of $z$) equals $\bar{z}$ on $\Gamma$  and is single-valued in the region outside the ellipse,
and also inside $\Gamma$ up to the slit $[-c,c]$. Hence this branch is the actual Schwarz function of the curve $\Gamma$.
The other branch, $S_+(z)$, is also single-valued in $\C\setminus [-c,c]$, and it analytically continues through the slit, 
to join $S_-(z)$ on the other side, compare Figure~\ref{fig:ellipse}. However it does not agree with $\bar{z}$ on $\Gamma$.
As $z\to\infty$ both branches tend to infinity linearly. 

It follows that $S(z)$ is an algebraic function which becomes single-valued on a two-sheeted Riemann surface $M$ over
$\C\cup\{\infty\}$ with branch points over $\pm c$, and with each branch having a simple pole over $z=\infty$. Moreover, it uniformizes the 
algebraic equation $P(z,w)=0$, where
\begin{equation}\label{Pzw}
P(z,w)=z^2+w^2-\frac{2(a^2+b^2)}{a^2-b^2}zw+\frac{4a^2b^2}{a^2-b^2},
\end{equation}
in the sense that $P(z,S(z))=0$ holds identically.
 
The Riemann surface $M$ is compact, has genus zero and has no singular points.
Therefore $M$ can be directly identified with the locus  $P(z,w)=0$,  provided one compactifies the latter.
This is to be done in two dimensional projective space, i.e. one introduces the homogenization 
$P(t,z,w)=t^2P(z/t,w/t)$ of $P$ and defines
$$
{\rm loc\,}P=\{[\,t:z:w\,]\in\C\P^2: P(t,z,w)=0\}.
$$

Having genus zero, $M$ can also be identified with the Riemann sphere, and so we have several different models:
$$
{\rm loc\,}P\cong M\cong S^2\cong\C\P^1\cong\C\cup\{\infty\}.
$$
On ${\rm loc\,}P$ (or the finite part of it) the coordinate functions $z$ and $w$ appearing in (\ref{Pzw})
turn into meromorphic functions $f$ and $f^*$ on defined on $M$ and satisfying $P(f,f^*)=0$. 
Here $M$ is basically an abstract Riemann surface, but choosing the model
$\C\cup\{\infty\}=\D\cup\partial\D\cup\D^e$
for $M=M_+\cup \Gamma \cup M_-$, one choice is to take $f$ to be the Joukowski map 
\begin{equation}\label{fellipse}
f(\zeta)=\frac{1}{2}\big((a-b)\zeta+(a+b)\zeta^{-1}\big).
\end{equation}
This function is univalent inside the unit disk $\D$ and maps it onto the exterior $\Omega=E^e$ of the ellipse. 
The holomorphically reflected function, defined in terms of $f$ by $f^*=\overline{f\circ J}$ where $J(\zeta)=1/\bar{\zeta}$, is
\begin{equation}\label{fstarellipse}
f^*(\zeta)=\frac{1}{2}\big((a+b)\zeta+(a-b)\zeta^{-1}\big).
\end{equation} 
It is easy to interpret (\ref{fellipse}) on the boundary: on taking $\zeta=e^{\I \varphi}\in\partial\D$ it gives the parametrization
$f(e^{\I \varphi})=a\cos \varphi -\I b \sin\varphi$
of $\Gamma$. The minus sign appears because decreasing $\varphi$ is to correspond to positive orientation of $\partial\Omega$.


As mentioned, the function $f(\zeta)$ is univalent in the unit disk, and it remains so up to the radius $R=\sqrt{\frac{a+b}{a-b}}=\frac{a+b}{c}>1$.
The annulus $1<|\zeta|<R$ is mapped onto $E\setminus [-c,c]$, and $f$ is in addition
univalent in $R<|\zeta|<\infty$, mapping that region onto $\C\setminus [-c,c]$. 
Thus $E\setminus [-c,c]$ is covered twice, and the circle $|\zeta|=R$ is mapped onto the focal segment $[-c,c]$,
which is also a branch cut for the inverse map.  On this circle $|\zeta|=R$ we have
$$
f(\pm\I\frac{a+b}{c})=0, \quad f(\pm\frac{a+b}{c})=\pm c, \quad f^\prime(\pm\frac{a+b}{c})=0.
$$
We notice that, in the notations of (\ref{ddff}), the zeros $\omega_j$ and poles $\xi_j$ of $f$ are given by
$$
\omega_j=\pm\I\frac{a+ b}{c}, \quad \xi_j = 
\begin{cases}
0\\
\infty
\end{cases}
\quad(j=1,2).
$$
As for the mapping properties we can write, symbolically,
$$
f(\D)=1\cdot E^e=\Omega,
$$
$$
f(\D^e)= 2\cdot E+ 1\cdot E^e,
$$
see Figure~\ref{fig:ellipse}.


\begin{figure}
\begin{center}
\includegraphics[width=\textwidth, scale=1, trim= 120 400 10 200]{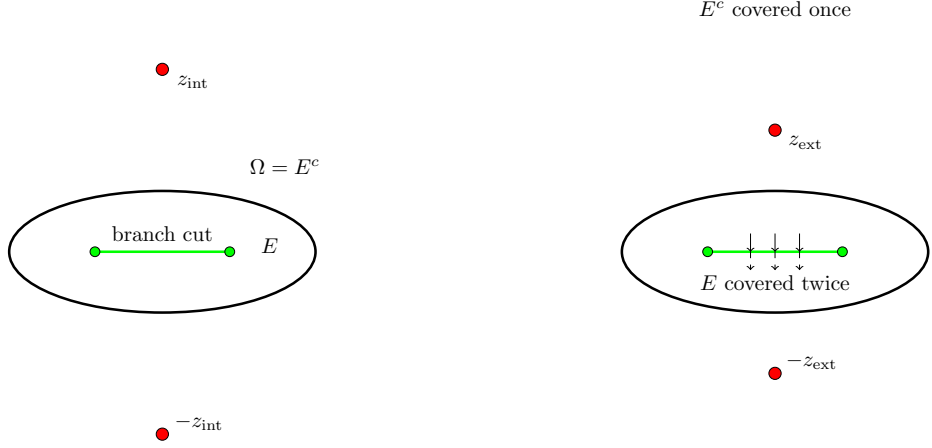}
\end{center}
\caption{{\it Left:} the exterior $\Omega$ of an ellipse $E$ as a two point quadrature domain. {\it Right:} similarly for the partly two-sheeted domain
$2\cdot E+ 1\cdot E^c$. Transition between the sheets via branch cut.}
\label{fig:ellipse}
\end{figure}


\subsection{Quadrature data}

We proceed to identify the quadrature identities (\ref{qisf}), or (\ref{sfqi1}), (\ref{sfqi2}), associated with the ellipse. 
With $f$ as in (\ref{fellipse}) and using directly (\ref{sfqi2}), we obtain
\begin{equation}\label{qisfellipse}
\frac{1}{2\pi \I} \int_{\D} h\,\frac{d\bar{f}\wedge d{f}}{(1+|f|)^2} =-\sum_{1+ff^*=0} {\rm Res\,} \Big(\frac{h  }{1+ff^*}\cdot\frac{df}{f}\Big).
\end{equation}
By (\ref{fellipse}), (\ref{fstarellipse}) we have
$$
1+f(\zeta)f^*(\zeta)=\frac{1}{4\zeta^2}\big(c^2 \zeta^4+(4+2(a^2+b^2))\zeta^2 +c^2\big), 
$$
hence the equation $1+ff^*=0$ for the quadrature nodes becomes
$$
c^4 \zeta^4+(4+2(a^2+b^2))c^2\zeta^2 +c^4=0.
$$
Solving first for $c^2\zeta^2$ gives two solutions
\begin{equation}\label{c2ellipse}
(c^2\zeta^2)_\pm=- ({2+a^2+b^2)\pm2\sqrt{(1+a^2)(1+b^2)}},
\end{equation}
which are negative real numbers satisfying
\begin{equation}\label{c4}
(c^2\zeta^2)_+\cdot(c^2\zeta^2)_-=c^4, 
\end{equation}
$$
(c^2\zeta^2)_+ +(c^2\zeta^2)_-=-4-2(a^2+b^2). 
$$
The signs $\pm$ correspond to the signs of the square root in (\ref{c2ellipse}).  

Taking next an exterior square root gives four purely imaginary solutions,
$\pm\zeta_{\rm ext}$ and $\pm\zeta_{\rm int}$, for $\zeta$ itself. They are 
\begin{equation}\label{zetapmext}
\pm\zeta_{\rm ext}=\pm\frac{\I}{c} \sqrt{2+a^2+b^2+2\sqrt{(1+a^2)(1+b^2)}}\in \D^e,
\end{equation}
\begin{equation}\label{zetapm}
\pm\zeta_{\rm int}=\pm\frac{\I}{c} \sqrt{2+a^2+b^2-2\sqrt{(1+a^2)(1+b^2)}}\in \D,
\end{equation}
with $\zeta_{\rm int}\cdot\zeta_{\rm ext}=\pm 1$ by (\ref{c4}).
The right member of (\ref{qisfellipse}) becomes
$$
-\sum_{1+ff^*=0} {\rm Res\,} \Big(\frac{h  }{1+ff^*}\cdot\frac{df}{f}\Big)
=-\sum_{1+ff^*=0} \frac{h  }{(1+ff^*)^\prime}\cdot\frac{f^\prime}{f}
$$
\begin{equation}\label{qipreliminary}
=-\sum_{\zeta=\pm \zeta_{\rm int}}\frac{h(\zeta)}{c^2 \zeta^2 +2+a^2+b^2}
\cdot \frac{c^2\zeta^2-(a+b)^2}{c^2\zeta^2+(a+b)^2},
\end{equation}
which in principle gives
\begin{equation}\label{qiellipse}
\frac{1}{2\pi \I} \int_{\D} h\,\frac{df\wedge d\bar{f}}{(1+|f|)^2} =A_{\rm int}\big(h(\zeta_{\rm int})+h(-\zeta_{\rm int})\big)
\end{equation}
with $\zeta_{\rm int}$ as in (\ref{zetapm}). Similarly, using the exterior nodes (\ref{zetapmext}),
$$
\frac{1}{2\pi \I} \int_{\D^e} h\,\frac{df\wedge d\bar{f}}{(1+|f|)^2} =A_{\rm ext}\big(h(\zeta_{\rm ext})+h(-\zeta_{\rm ext})\big).
$$

The coefficients $A_{\rm int}$ and $A_{\rm int}$ are obtained by inserting (\ref{zetapmext}), (\ref{zetapm})
(or more directly (\ref{c2ellipse})) in (\ref{qipreliminary}). Alternatively, (\ref{intres}) or (\ref{intres1})  can be used. 
The result is, after simplifications,
\begin{equation}\label{Aint}A_{\rm int}=\frac{1}{2}\Big(1-\frac{ab}{\sqrt{(1+a^2)(1+b^2)}}\Big),
\end{equation}
\begin{equation}\label{Aext}
A_{\rm ext}=\frac{1}{2}\Big(1+\frac{ab}{\sqrt{(1+a^2)(1+b^2)}}\Big).
\end{equation}
In particular,
\begin{equation}\label{Aextint}
A_{\rm int}+A_{\rm ext}=1,
\end{equation}
which can be seen as a direct consequence of (\ref{intm}) with $m=2$.

The points $\pm\zeta_{\rm int}$ and $\pm\zeta_{\rm ext}$ have no particular geometric meaning themselves, 
it is rather the image points $\pm z_{\rm int}=f(\pm \zeta_{\rm int})$ and $\pm z_{\rm ext}=f(\pm \zeta_{\rm ext})$
which can be interpreted geometrically. They are most easily obtained by solving directly the equation $1+zS(z)=0$ 
(relevant in (\ref{intres1})), which corresponds to $1+f(\zeta)f^*(\zeta)=0$. That equation becomes
\begin{equation}\label{c4z}
c^4z^4+\big(2(a^2+b^2)+4a^2b^2\big)c^2z^2+c^4=0,
\end{equation}
and the solutions are
\begin{equation}\label{zint}
\pm z_{\rm int}=\pm\frac{\I}{c}\sqrt{2a^2b^2+a^2+b^2 + 2ab\sqrt{(1+a^2)(1+b^2)} },
\end{equation}
\begin{equation}\label{zext}
\pm z_{\rm ext}=\pm\frac{\I}{c}\sqrt{2a^2b^2+a^2+b^2 - 2ab\sqrt{(1+a^2)(1+b^2)} }.
\end{equation}
They have symmetries similar to those for $\zeta$, for example $z_{\rm int}\cdot z_{\rm ext}=\pm 1$,
and are located along the imaginary axis in the order 
$$
\im z_{\rm int}<-1<\im z_{\rm ext}<0<-\im z_{\rm ext}<1<-\im z_{\rm int}.
$$

In this physical $z$-plane the quadrature identity (\ref{qiellipse})  becomes
\begin{equation}\label{qiellipsephysical}
\frac{1}{2\pi \I} \int_{\Omega} h(z)\,\frac{d\bar{z}\wedge d{z}}{(1+|z|)^2} =A_{\rm int}\big(h(z_{\rm int})+h(-z_{\rm int})\big),
\end{equation}
in terms of the above data  and with $h$ considered as a function of $z=f(\zeta)$.
The identities, (\ref{qiellipse}) and (\ref{qiellipsephysical}), are
somewhat more precise versions of formulas appearing in \cite{Gustafsson-Tkachev-2011}. 


\section{The hyperbola}\label{sec:hyperbola}

\subsection{General}
 
We consider next a general hyperbola on the standard form
\begin{equation}\label{hyperbola0}
\Gamma:\quad\frac{x^2}{a^2}-\frac{y^2}{b^2}=1,
\end{equation}
obtained from the equation (\ref{ellipse}) for the ellipse by replacing $b$ with $\I b$. Thus most of the computations will be the same
as those for the ellipse, but the geometric considerations and certain other issues will be different, and for that reason we prefer to be 
fairly explicit with the details. In addition we consider the hyperbola to be the main new object of investigation for the present paper.
For the ellipse we assumed $a>b$, mainly because one usually visualizes the ellipse as oriented along the $x$-axis. 
For the hyperbola we assume only $a,b>0$. 

In complex coordinates the equation (\ref{hyperbola0}) becomes $P(z,\bar{z})=0$, where
\begin{equation}\label{hyperbola}
P(z,w)=z^2+{w}^2-\frac{2(a^2-b^2)}{a^2+b^2}zw-\frac{4a^2b^2}{a^2+b^2}.
\end{equation}
Solving for $\bar{z}$ gives the Schwarz function
\begin{equation}\label{Schwarzellipse}
S(z)=\frac{a^2-b^2}{c^2}\,z \pm \frac{2 ab}{c^2}\sqrt{c^2-z^2},
\end{equation}
where now $c=\sqrt{a^2+b^2}>0$. The two branch points $\pm c$ are located on the real axis, 
and branch cuts, not intersecting the curve, can be chosen to consist of the two segments $[c,+\infty)$, $(-\infty,-c]$, which meet
at the point of infinity of the Riemann sphere. 

The above means that $S(z)$ has two single-valued branches near the origin. That branch which satisfies $S(0)=2ab/c$
analytically continues to $\Gamma$ and satisfies $S(z)=\bar{z}$ on $\Gamma$. So this branch is the actual Schwarz function of $\Gamma$.
One might then expect that a null quadrature identity similar to that for the exterior of an ellipse, for example (\ref{nullqd}), 
would hold in that component of $\C\setminus \Gamma$ which is free from branch points.
But it does not, and one naturally wonders why. In order to produce an answer we need to make a fairly systematic study of quadrature properties for
hyperbola domains. 
We name the components of $\C\setminus \Gamma$ according to
$$
\Omega_0=\text{that component of } \C\setminus \Gamma \text{ which contains the point }z=0, \,\,\,\,
$$
$$
\Omega_+=\text{that component of } \C\setminus \Gamma \text{ which contains the point }z=+c,
$$
$$
\Omega_-=\text{that component of } \C\setminus \Gamma \text{ which contains the point }z=-c.
$$
Thus $S(z)$ has single-valued branches in $\Omega_0$, but not in $\Omega_\pm$.

When trying to complete the above components by adding points of infinity one really has to make 
a choice in which space to make the completion: in $\C\P^1$ or in $\R\P^2_{\rm red}$.
In the first case there is only one point of infinity, and after an inversion 
the hyperbola becomes a lemniscate and looks like the figure ``$\infty$''. Then  $\Omega_0$ is mapped to the exterior component,
and the $\Omega_\pm$ are mapped to the insides of the two small loops (in ``$\infty$''). See further Section~\ref{sec:lemniscate}.

In the second case the curve $\Gamma$ itself completes into the set of those $[\,t: z\,]\in \R\P^2_{\rm red}$  satisfying
\begin{equation}\label{hyperbola1}
z^2+\bar{z}^2-\frac{2(a^2-b^2)}{a^2+b^2}z\bar{z}-\frac{4a^2b^2}{a^2+b^2}t^2=0.
\end{equation} 
Choosing $t=1$ gives the finite part of the hyperbola, namely (\ref{hyperbola}),
while the additional points of infinity are of the form $[\,0:e^{\I \varphi}\,]$
with $\varphi\in \R$ satisfying
\begin{equation}\label{cosphi}
\cos 2\varphi=\frac{a^2-b^2}{a^2+b^2}.
\end{equation}
This equation defines the asymptotic directions of the hyperbola.
Note that the left member has period $\pi$, which means that the two directions of an asymptotic line
correspond to the same point of infinity in the limit. In particular this means that the two components
of $\Gamma$ do not intersect at the projective infinity, but are pieces of one and the same smooth curve
in $\R\P^2_{\rm red}$. If the component of $\Gamma$ which bounds $\Omega_+$ is oriented so that $\Omega_+$ lies to the
left, then this curve continues with the other component of $\Gamma$ oriented so that $\Omega_0$ lies to the left, see
Figure~\ref{fig:hyperbola2}. 

This orientation of $\Gamma$, in the direction of decreasing values of $y$, is consistent with the idea that $\Gamma$
is the boundary of the multi-sheeted domain
\begin{equation}\label{multipleOmega}
\Omega=2\cdot \Omega_+ +1\cdot \Omega_0,
\end{equation}
which may be considered as the image of $M_+$ under a suitable covering map $f:M\to \C\cup\{\infty\}$.
An explicit choice of $f$ can be obtained by simply changing $b$ to $\I b$ in the Joukowski map we had in the ellipse case,
see (\ref{fellipse}), (\ref{fstarellipse}) and subsequent equations.
This gives, for $\zeta\in \C\cup \{\infty\}$,
\begin{equation}\label{fhyperbola}
f(\zeta)=\frac{1}{2}\big((a-\I b)\zeta+(a+\I b)\zeta^{-1}\big), 
\end{equation}
with an accompanying function $f^*(\zeta)$ as in (\ref{fstarellipse}), namely
$$
f^*(\zeta)=\frac{1}{2}\big((a+\I b)\zeta+(a-\I b)\zeta^{-1}\big). 
$$
With $P(z,w)$ as in (\ref{hyperbola}) the equation $P(f(\zeta),f^*(\zeta))=0$ holds identically, exactly as in the ellipse case.

The function $f^*$ is still a holomorphic reflection of $f$, i.e. $f^*=\overline{f\circ J}$, however now with a different
involution, namely $J(\zeta)=\bar{\zeta}$.
This means that the fixed point set of $J$ is the real axis and the decomposition $M=M_+\cup \Gamma\cup M_-$
of the uniformizing Riemann surface becomes modeled by $\C\cup\{\infty\}=\U\cup (\R\cup\{\infty\}) \cup \L$.
Here $\U$ and $\L$ are the upper and lower half planes, respectively. Thus we take $M_+=\U=\{\zeta\in\C :\im \zeta>0\}$. 
The zeros and poles of $f$ are given by
$$
\omega_j=\pm\I\frac{a+\I b}{c}, \quad \xi_j = 
\begin{cases}
0\\
\infty
\end{cases}
\quad (j=1,2)
$$
with notations as in (\ref{ddff}). Thus $\omega_j\in \partial \D$,  $\xi_j\in\partial \U$.
The zeros $\zeta_\pm=\pm (a+\I b)/c$ of the derivative $f^\prime$ are mapped by $f$ onto the focal points $\pm c$,
which hence are branch points for $f$ considered as a covering map.

\begin{remark}
One may wonder why the involution changes when one replaces $b$ by $\I b$, i.e. when moving from ellipse to hyperbola.
With $P(z,w)$ as in (\ref{Pzw}) and $f$ and $f^*$ as in (\ref{fellipse}) and (\ref{fstarellipse}), the algebraic  identity $P(f,f^*)=0$
holds independent of what type of values $a$ and $b$ take. But when insisting on having the relationship $f^*=\overline{f\circ J}$
the meaning of $J$ has to change because of the action of complex conjugation. 

As a further step one may take both of $a$ and $b$ to be imaginary, i.e. to change them, in (\ref{ellipse}), (\ref{Pzw}), (\ref{fellipse}),
(\ref{fstarellipse}), to $\I a$ and $\I b$ respectively. Then $P(f,f^* )=0$ still holds, but now we can interpret $f^*=\overline{f\circ J}$ 
only with $J(\zeta)=-{1}/{\bar{\zeta}}$. This involution has no fixed points, as is consistent with the fact that the curve
(\ref{ellipse}) in this case has no real locus.
\end{remark}


\begin{figure}
\begin{center}
\includegraphics[width=\textwidth, scale=0.8, trim=150 300 0 200]{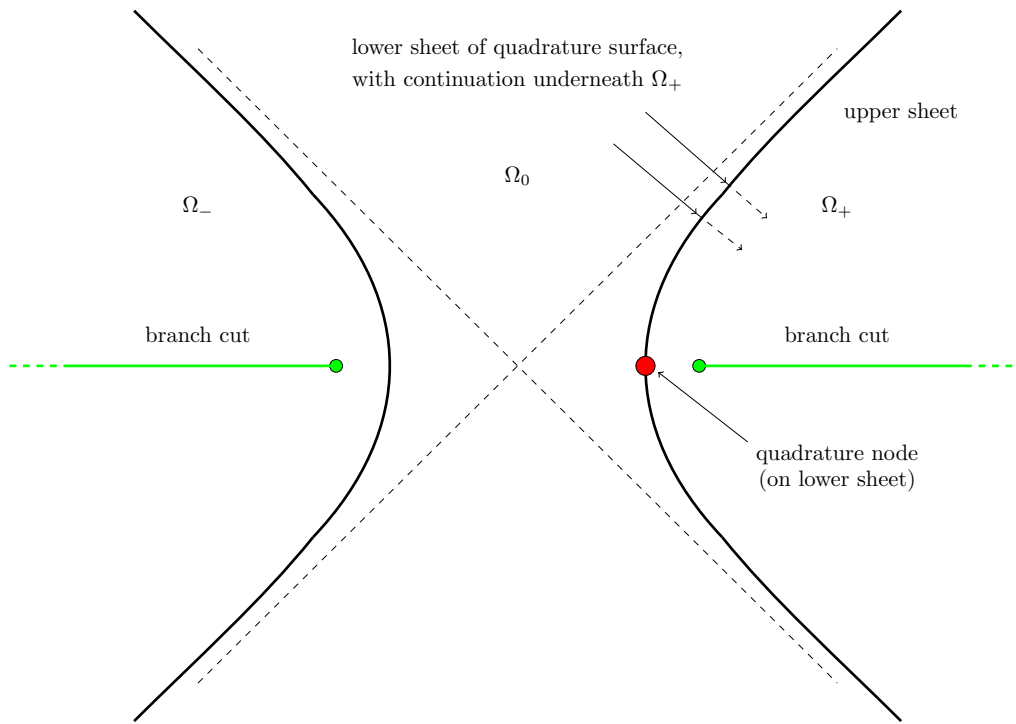}
\end{center}
\caption{The hyperbola $x^2-y^2=1$ as part of the Riemann sphere,
with the single quadrature node (of order two) located on the curve itself, however
on a diffeent sheet (as interior point of $\Omega_0$ and its continuation under $\Omega_+$). }
\label{fig:hyperbola3}
\end{figure}


\subsection{Quadrature data}

According to (\ref{sfqi2}) the quadrature nodes come from the solutions $\zeta\in \U$ of $1+f(\zeta)f^*(\zeta)=0$. 
This equation spells out to
$$
c^4\zeta^4+\big(4+2(a^2-b^2)\big)c^2\zeta^2+c^4=0,
$$
whereby
\begin{equation}\label{c2}
(c^2\zeta^2)_\pm=-(2+a^2-b^2)\pm 2\sqrt{(1+a^2)(1-b^2)}
\end{equation}
giving in total four solutions in $\C$,
\begin{equation}\label{zetapmh}
c\,\zeta_{j}=\pm\sqrt{-(2+a^2-b^2)\pm 2\sqrt{(1+a^2)(1-b^2)}},
\end{equation}
($j=1,2,3,4$), corresponding to different choices of square roots. 
Clearly the symmetries $\zeta\mapsto \bar{\zeta}$ and $\zeta\mapsto -\zeta$  act on the set of solutions.
In addition, there are symmetries with respect to the unit circle since $\zeta_1\zeta_2\zeta_3\zeta_4=1$.
We aim at selecting square roots so that $\zeta_1, \zeta_2\in\U$
and $\zeta_3,\zeta_4\in\L$. 

The parameter value $b=1$ marks a phase change, and we have to consider three different regimes:
\begin{itemize}

\item  \,$0<b<1$: In this regime the right member of (\ref{c2}) is a negative real number for both choices of signs.
Indeed 
$$
(2+a^2-b^2)^2-\big(2\sqrt{(1+a^2)(1-b^2)}\big)^2=c^4>0.
$$
It follows that all roots $\zeta_j$ in (\ref{zetapm}) are purely imaginary, and that we can order them according to
$$
\im \zeta_4<\im \zeta_3<0<\im \zeta_1<\im \zeta_2. 
$$
This means that $\zeta_1$ and $\zeta_3$ correspond to choosing the plus sign in (\ref{c2}), $\zeta_2$ and
$\zeta_4$ to choosing the minus sign, and that the outer square root in (\ref{zetapmh}) is chosen so that $\zeta_1, \zeta_2\in \U$.   

\item  \,$b=1$: This case is similar, but with two double roots:
$$
\zeta_4=\zeta_3=-\I, \quad \zeta_1 =\zeta_2=\I.   
$$

\item \,$b>1$: Now we write (\ref{zetapmh}) as
\begin{equation}\label{zetapmh1}
c\,\zeta_{j}=\pm\sqrt{\alpha\pm\I\beta},
\end{equation}
where
\begin{align*}
\alpha&=-(2+a^2-b^2),\\
\beta&=2\sqrt{(a^2+1)(b^2-1)}>0,
\end{align*}
and we order the roots so that $j=1,3$
correspond to $\pm\sqrt{\alpha+\I\beta}$ with $\zeta_1\in\U$, $\zeta_3\in \L$, and $j=2,4$ correspond to $\pm\sqrt{\alpha-\I\beta}$
with $\zeta_2\in \U$, $\zeta_4\in\L$. Since $\beta>0$ this means that $\zeta_j$ is in the $j$:th quadrant for each $j=1,2,3,4$.
\end{itemize}
 
With the above agreements  the identity (\ref{sfqi2}) becomes, when $b\ne 1$,
$$
\frac{1}{2\pi \I} \int_{\U} h\,\frac{d\bar{f}\wedge d{f}}{(1+|f|)^2} 
=-\sum_{1+ff^*=0}{\rm Res\,}\Big(\frac{h  }{1+ff^*}\cdot\frac{df}{f}\Big)
$$
\begin{equation}\label{roughqi}
=-\sum_{j=1}^2\frac{h(\zeta_j)}{c^2\zeta_j^2+2+a^2-b^2}
\cdot \frac{c^2\zeta_j^2-(a+\I b)^2}{c^2\zeta_j^2+(a+\I b)^2}.
\end{equation}
This is in principle of the form
\begin{equation}\label{qihyperbolaU}
\frac{1}{2\pi \I} \int_{\U} h\,\frac{d\bar{f}\wedge d{f}}{(1+|f|)^2} =A_1h(\zeta_1)+A_2h(\zeta_2).
\end{equation}
The quadrature identity associated to the lower half plane is of the same form:
\begin{equation}\label{qihyperbolaL}
\frac{1}{2\pi \I} \int_{\L} h\,\frac{d\bar{f}\wedge d{f}}{(1+|f|)^2} =A_3h(\zeta_3)+A_4h(\zeta_4).
\end{equation}

In similarity with the ellipse case, insertion of (\ref{c2}) in (\ref{roughqi}) gives only two possible values, $A_+$ and $A_-$, of the coefficients $A_j$.
In the hyperbola case these need not be real, and to exhibit the situation clearly we have to separate the cases $b<1$ and $b>1$. 
After some computations one gets
\begin{equation}\label{Apm1}
A_\pm=\frac{1}{2}\Big(1\pm\frac{\I ab}{\sqrt{(1+a^2)(1-b^2)}}\Big) \quad \text{when }b<1,
\end{equation}
\begin{equation}\label{Apm2}
A_\pm=\frac{1}{2}\Big(1\pm\frac{ab}{\sqrt{(a^2+1)(b^2-1)}}\Big) \quad \text{when }b>1.
\end{equation}
In both cases
$$
A_++A_-=1,
$$
and, for the difference,
\begin{equation}\label{AminusA}
A_+-A_-=
\frac{\I ab}{\sqrt{(1+a^2)(1-b^2)}}\quad \text{if } b<1,
\end{equation}
\begin{equation}\label{AminusA2}
A_+-A_-=\frac{ab}{\sqrt{(a^2+1)(b^2-1)}}\quad \text{if } b>1,
\end{equation}
with positive square roots in both cases.  

A detailed analysis of the above formulas leads to the following identifications, for (\ref{qihyperbolaU}) and (\ref{qihyperbolaL}):
\begin{equation}\label{AjApm}
A_1=A_3=A_-, \quad A_2=A_4=A_+.
\end{equation}
As a consequence, the main quadrature identity (\ref{qihyperbolaU}) takes the form
$$
\frac{1}{2\pi \I} \int_{\U} h\,\frac{d\bar{f}\wedge d{f}}{(1+|f|)^2} =
\frac{1}{2}\big(h(\zeta_1)+h{(\zeta_2)}\big)
-\frac{\I ab}{2\sqrt{(1+a^2)(1-b^2)}}\big(h(\zeta_1)-h(\zeta_2)\big) 
$$
when $b<1$, 
$$
\frac{1}{2\pi \I} \int_{\U} h\,\frac{d\bar{f}\wedge d{f}}{(1+|f|)^2} =
\frac{1}{2}\big(h(\zeta_1)+h{(\zeta_2)}\big)
-\frac{ab}{2\sqrt{(a^2+1)(b^2-1)}}\big(h(\zeta_1)-h(\zeta_2)\big) 
$$
when $b>1$. 
The first terms in the left member are quite natural since the partly double-sheeted image domain of $\U$ is to
occupy as much area as the Riemann sphere itself (the same for the image domain of $\L$). The second term
is purely imaginary in the case $b<1$,  and a negative real number in the case $b>1$. Note that it be made arbitrarily 
big by choosing $b$ to be sufficiently close to one. Also, $b>1$ can be adapted so that the contribution from $h(\zeta_1)$
disappears, see further Section~\ref{sec:loss of weight}. 

Turning to the special case $b=1$, we have $\zeta_1=\zeta_2=\I$ and the quadrature identity becomes
\begin{equation}\label{qihyperbolab}
\frac{1}{2\pi \I} \int_{\U} h\,\frac{d\bar{f}\wedge d{f}}{(1+|f|)^2} =h(\I)-\frac{a}{{1+a^2}}h^\prime(\I). 
\end{equation}
One can show this by letting $\zeta_1,\zeta_2\to \I$ (whereby $b\to 1$)  in the right member of (\ref{qihyperbolaU}) and using 
(\ref{zetapmh}), (\ref{AminusA}), (\ref{AminusA2}), (\ref{AjApm}) when passing to the limit:
$$
A_1h(\zeta_1)+A_2h(\zeta_2)=\frac{1}{2}\big(A_1+A_2\big)\big(h(\zeta_1)+h(\zeta_2)\big)
+\frac{1}{2}\big(A_1-A_2\big)\big(h(\zeta_1)-h(\zeta_2)\big)
$$
$$
=\big(A_++A_-\big)h(\zeta_{1,2})
-\frac{1}{2}\big(A_+-A_-\big)(\zeta_1-\zeta_2)\cdot\frac{h(\zeta_1)-h(\zeta_2)}{\zeta_1-\zeta_2}
\to h(\I)-\frac{a}{1+a^2}h^\prime (\I).
$$


\subsection{Quadrature identities in the physical plane}\label{sec:qiphysical}

Returning to the general hyperbola, and considering it now  as a partly double-sheeted domain in the $z$-plane, 
the quadrature nodes $z_j=f(\zeta_j)$  in this picture are most easily obtained by solving the equation $1+zS(z)=0$,
namely (\ref{c4z}) with $b$ replaced by $\I b$. Thus (compare (\ref{zint}), (\ref{zext})),
$$
z_j=\pm\frac{1}{c}\sqrt{2a^2b^2-a^2+b^2 \pm 2ab\sqrt{(a^2+1)(b^2-1)} }.
$$
The analysis of the square roots is similar to what we had in the previous section: the outer square root selects between the two pairs
$\{z_1,z_2\}$ and $\{z_3,z_4\}$ and the inner square root makes selections within each of them. The result is the 
following classification (compare Figure~\ref{fig:hyperbola14}).

\begin{itemize}
\item  \,$0<b<1$: The nodes $z_j$ consist of two complex conjugate pairs located on the  unit circle
with $z_1,z_2,z_3,z_4$ in, respectively, the fourth, first, second and third quadrant. 

\item  \,$b=1$: $z_1=z_2= 1$, $z_3=z_4=-1$.

\item \,$b>1$: All nodes are real and satisfy 
$$
z_3<-1<z_4<0<z_2<1<z_1.
$$
\end{itemize}

The quadrature identity is in all cases of the form
\begin{equation}\label{qihyperbolaphysical}
\frac{1}{2\pi \I} \int_{\Omega} h(z)\,\frac{d\bar{z}\wedge d{z}}{(1+|z|)^2} =A_-h(z_{1})+A_+h(z_{2})
\end{equation}
with $A_\pm$ as in (\ref{Apm1}), (\ref{Apm2}). The integration takes place on a two-sheeted and branched Riemann surface
and the function $h$ is allowed to take different values on the two sheets, these being joined by branch points at $\pm c$. 

When $b=1$, with $z_1=z_2=1$ and a derivative of $h$ appearing in the quadrature identity, see (\ref{qihyperbolab}), 
one has to take the chain rule into account when  interpreting that derivative as a derivative with respect
to $z$: since $\frac{d}{d\zeta}h(f(\zeta))=f^\prime(\zeta)h^\prime(z)$ a factor 
$f^\prime(\I)=a$ has to be included. The result is
\begin{equation}\label{qihhprime}
\frac{1}{2\pi \I} \int_{\Omega} h(z)\,\frac{d\bar{z}\wedge d{z}}{(1+|z|)^2} =h(1)-\frac{a^2}{{1+a^2}}h^\prime(1).
\end{equation}

Specializing further to $a=b=1$, the unique quadrature node $z_1=z_2=1$ for $\Omega$
is in this case located on the curve itself. The node belongs actually to the interior of that  sheet of
domain $\Omega$ to which $z_1=z_2$ belongs, but happens to simultaneously be on the boundary of the other sheet. 
See Figure~\ref{fig:hyperbola3}. 


\begin{figure}
\includegraphics[scale=0.35, trim=100 200 0 150]{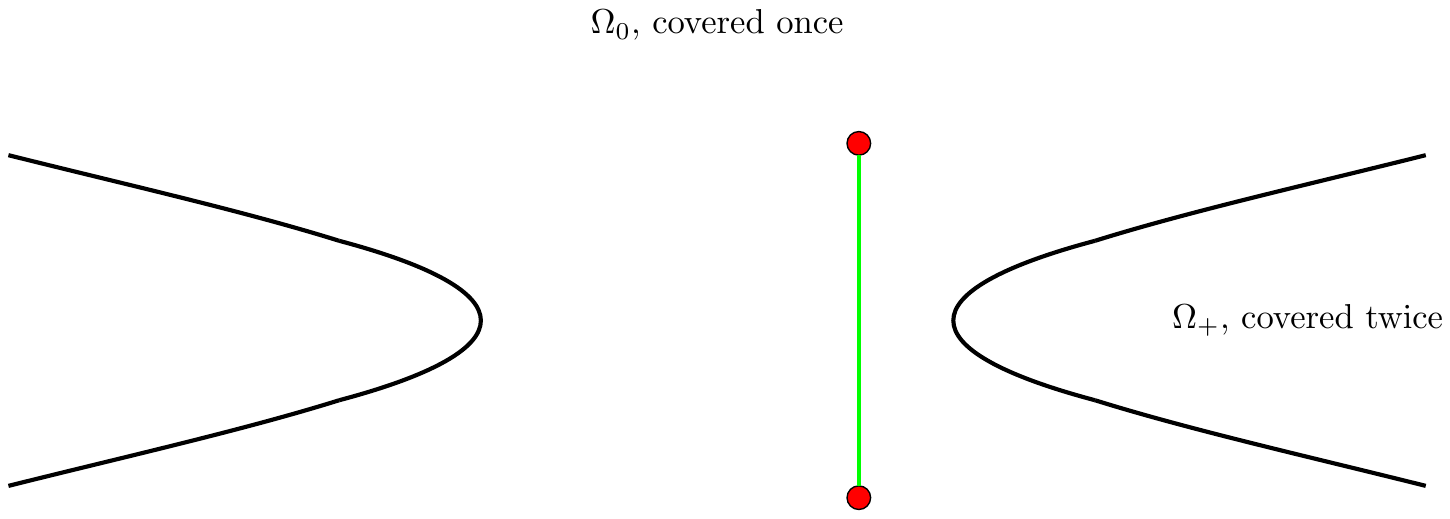}
\includegraphics[scale=0.35, trim=0 120 10 150]{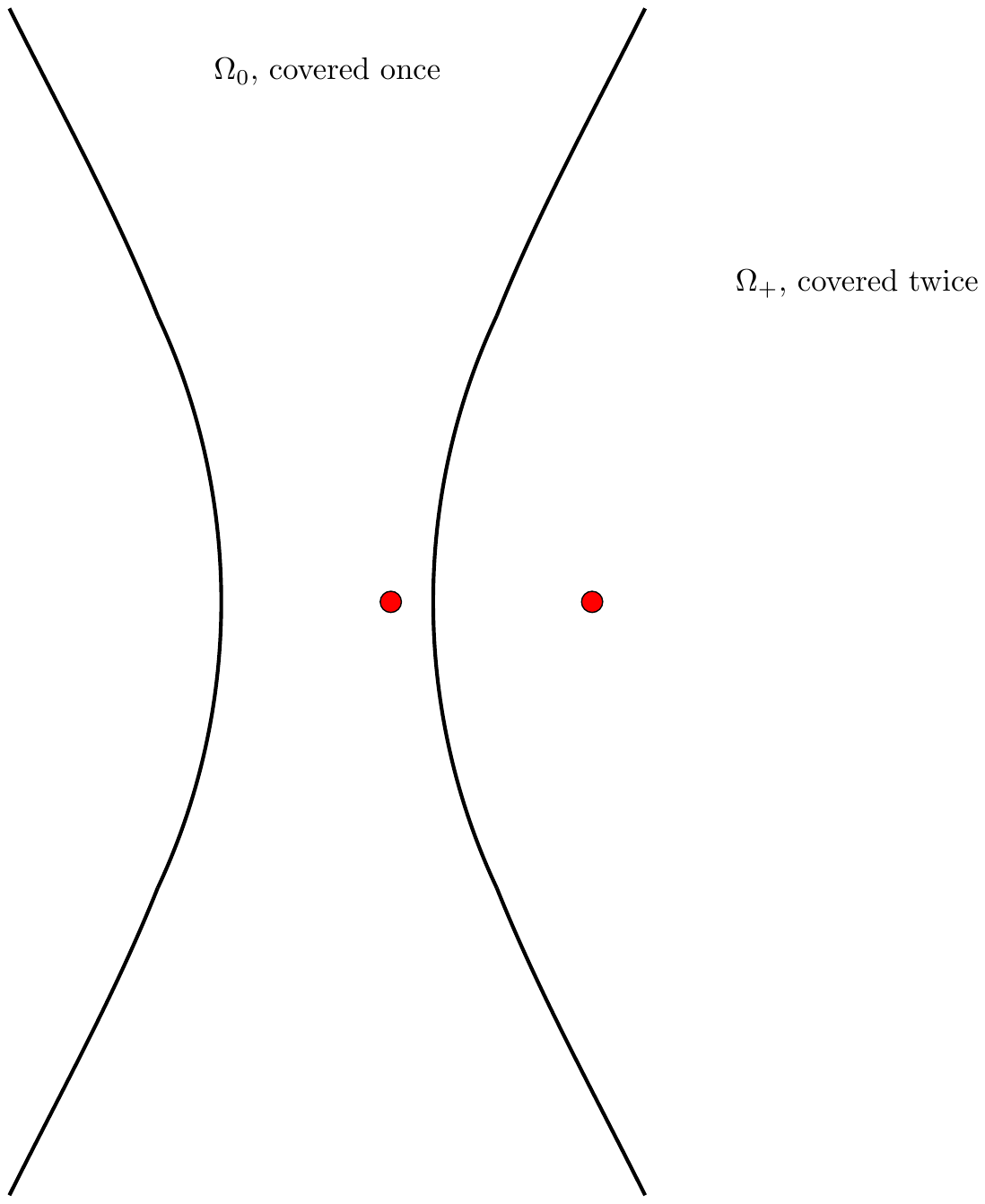}
\caption{Quadrature properties of hyperbola. {\it Left}: $a>b$,  $b<1$ and non-real weights, or else dipole density along green line. 
{\it Right:} $a<b$, $b>1$ and real weights.}
\label{fig:hyperbola14} 
\end{figure}


\subsection{Loss of weight}\label{sec:loss of weight}

Considering in some more detail the case $b>1$, a perhaps astonishing phenomenon is that the coefficient for $h(z_1)$ 
in (\ref{qihyperbolaphysical}) may vanish. 
In fact, given any $a>0$, the coefficient $A_-$ vanishes if
$ab=\sqrt{(a^2+1)(b^2-1)}$, which occurs for $b=\sqrt{1+a^2}$. See equation (\ref{Apm2}). The quadrature nodes are then, in the parameter space, 
given by $c\,\zeta_j=\pm\I\sqrt{1\mp 2\I ab}$, the order of the inner signs being coupled to those in $A_\pm$. 
The nodes in the $z$-plane become 
$$
z_j=f(\zeta_j)=
\begin{cases}
c \quad &\text{for } j=1,\\
1/c \quad &\text{for } j=2.
\end{cases}
$$
 
Thus the node $z_1=c$ looses its weight and becomes only  ``virtual''. The quadrature identity reduces to
\begin{equation}\label{onepointqi}
\frac{1}{2\pi \I} \int_{\Omega} h\,\frac{d\bar{z}\wedge d{z}}{(1+|z|)^2} =\frac{1}{2}\Big(1+\frac{ab}{\sqrt{(a^2+1)(b^2-1)}}\Big)  h(\frac{1}{c}).
\end{equation}
An important observation is that $c$ is a branch point. In other words, $f^\prime(\zeta_1)=0$ in the above notation, and 
since $d\bar{f}\wedge df/(1+|f|^2)^2$ contains the factor $|f^\prime|^2$ this makes functions $h$ with a light singularity be integrable 
when lifted to the covering surface. Slightly singular functions could therefore in principle be accepted as test functions. 
In that sense the loss of a quadrature node can be explained by shortage of test functions, since we actually  allow only those which are holomorphic. 

The above phenomenon, loss of weight for quadrature nodes at branch points, 
has been discussed also in \cite{Gustafsson-Lin-2021} (see in particular Section~8.4.1 there), and it 
is responsible for crucial properties of contractive zero divisors in Bergman space
\cite{Hedenmalm-1991, Hedenmalm-Korenblum-Zhu-2000}. 

\subsection{Non-real coefficients and dipole densities}

When $b>1$, the coefficients $A_j=A_\pm$ in (\ref{qihyperbolaphysical}) are real numbers and the quadrature
identity holds for the real and imaginary parts of $h$ independently of each other. This means that 
(\ref{qihyperbolaphysical}) actually holds for harmonic test functions $h$. However, when $b<1$ the coefficients $A_\pm$
are non-real and the identity  (\ref{qihyperbolaphysical}) mixes the real and imaginary parts of $h$. One can still obtain
a quadrature identity for harmonic functions, but then one must extend the setting and allow more general quadrature functionals in the right
member. To make this precise we rewrite (\ref{qihyperbolaphysical}), when $b<1$ and with $z_1=x_1+\I y_1$, as
$$
\frac{1}{2\pi \I} \int_{\Omega} h(z)\,\frac{d\bar{z}\wedge d{z}}{(1+|z|)^2} =\frac{1}{2}\big(h(z_{1})+h(\bar{z}_{1})\big)
-\frac{\I ab}{2\sqrt{(1+a^2)(1-b^2)}}\big(h(z_1)-h_2(\bar{z}_1)\big)
$$
$$
=\frac{1}{2}\big(h(z_{1})+h(\bar{z}_{1})\big)-\frac{\I ab}{2\sqrt{(1+a^2)(1-b^2)}} \cdot \int_{-y_1}^{y_1} \frac{\partial h(x_1+\I y)}{\partial y} dy
$$
$$
=\frac{1}{2}\big(h(z_{1})+h(\bar{z}_{1})\big)+\frac{ab}{2\sqrt{(1+a^2)(1-b^2)}} \cdot \int_{-y_1}^{y_1} \frac{\partial h(x_1+\I y)}{\partial x_1} dy.
$$
On the latter form there is no mixing between real and imaginary parts, hence the equation holds for harmonic functions $h$ 
without involving  their harmonic conjugates.
A potential theoretic interpretation of the equation is that it expresses gravi-equivalence between the spherical area measure restricted to $\Omega$
and two point sources at $z_1$ and $\bar{z}_1$ plus a line dipole density on the segment from $\bar{z}_1$ to $z_1$.

As $b\to 1$, whereby $y_1\to 0$, the formula agrees in the limit with (\ref{qihhprime}).
One has to keep in mind that $y_1<0$, and for that reason the final integral has size 
$\approx -2y_1\cdot \frac{\partial h(x_1)}{\partial x_1}$.


\subsection{Summary}

In summary, we have seen that multi-sheeted domains bounded by hyperbolas are indeed quadrature domains of order two
for the spherical area measure. If one of the quadrature nodes happens to coincide with a branch point the order actually decreases 
to one. Replacing spherical area measure by the Euclidean, the mentioned domains
become null quadrature domains (still multi-sheeted), as indicated in Remark~\ref{rem:varepsilon}. The explanation is that
the quadrature nodes move to infinity and that the holomorphic functions $h$ used as test functions have to vanish of order two at 
infinity to qualify as  test functions. Hence the quadrature nodes become invisible.

As for the forms of the quadrature identities we obtained the following identities in the parameter space:

\begin{proposition}\label{prop:qihyperbola}
The quadrature identity (\ref{sfqi1}) takes in the case of the partly double sheeted domain $\Omega$ in (\ref{multipleOmega})
bounded by the hyperbola (\ref{hyperbola0}), and in terms of the parametrization
$f:\U\to \Omega$ given by (\ref{fhyperbola}),
the specific form (\ref{qihyperbolaU}). Written in a different way,
$$
\frac{1}{2\pi \I} \int_{\U} h\,\frac{df\wedge d\bar{f}}{(1+|f|)^2} =\frac{1}{2}(h(\zeta_1)+h(\zeta_2))+\frac{1}{2}\frac{\I ab}{\sqrt{(1+a^2)(1-b^2)}}(h(\zeta_1)-h(\zeta_2)),
$$
when $b<1$, and
$$
\frac{1}{2\pi \I} \int_{\U} h\,\frac{df\wedge d\bar{f}}{(1+|f|)^2} =\frac{1}{2}(h(\zeta_1)+h(\zeta_2))+\frac{1}{2}\frac{ab}{\sqrt{(a^2+1)(b^2-1)}}(h(\zeta_1)-h(\zeta_2)),
$$
when $b>1$. The quadrature nodes $\zeta_j$ are given by (\ref{zetapmh}). When $b=1$, then $\zeta_1=\zeta_2=\I$, 
and the identity has the form (\ref{qihyperbolab}), i.e.
$$
\frac{1}{2\pi \I} \int_{\U} h\,\frac{d\bar{f}\wedge d{f}}{(1+|f|)^2} =h(\I)+\frac{a}{{1+a^2}}h^\prime(\I).
$$
In all above identities $h$ is assumed to be holomorphic in $\U$ and integrable with respect to the measure in the left member.
\end{proposition}


\section{The parabola}\label{sec:parabola}

We consider a general parabola on the form
$$
\Gamma: \quad 4ax=y^2, 
$$
or
\begin{equation}\label{parabola} 
8a(z+\bar{z})+(z-\bar{z})^2=0.
\end{equation}
As a complex curve in $\C\P^2$ with coordinates $[\,t:z:w\,]$ it is  
\begin{equation}\label{projective parabola}
8at(z+w)+(z-w)^2=0,
\end{equation}
and in real projective space $\R\P^2_{\rm red}$, with coordinates $[\,t:x+\I y\,]$,
$$
4atx=y^2.
$$

When viewed in $\R\P^2_{\rm red}$ the parabola is a smooth curve which passes through exactly one point of infinity, namely $[\,0: 1\,]$. 
At that point the parabola is tangent to the line of infinity.
Indeed, near $[\,0:1\,]$ the line of infinity can be parametrized as $u\mapsto [\,0:e^{\I u}\,]$, 
and the parabola by $v\mapsto [\,\frac{v^2}{4a}:1+\I v\,]$. Here $u,v$ are real, with $u=v=0$ corresponding to the point
$[\,0:1\,]$. The derivatives at $u=v=0$ are $[\,0:\I\,]$ in both cases, hence the two curves are tangent to each other.
In comparison, the hyperbola passes through two real points of infinity   
while the real ellipse curve does not reach infinity at all.
On the other hand, when viewing the parabola as a curve in $\C\cup \{\infty\}$, it becomes singular at the point of
infinity, having a cusp there. Indeed, the parabola becomes a cardioid when seen from the point of infinity, 
see Section~\ref{sec:cardioid}.

\begin{remark}\label{rem:infinities}
In the complex projective space $\C\P^2$, the parabola (\ref{projective parabola}) again meets the line of infinity only at one point, namely 
$[\,t:z:w\,]=[\,0: 1:1\,]$, which is a point of tangency. 
For the general ellipse (\ref{ellipse}) the points of infinity for the complexified boundary curve are $[\,0:a-b:a+b\,]$ and $[\,0:a+b:a-b\,]$. 
In the special case of a circle ($a=b=r$) these points become
$[\,0:0:1\,]$ and $[\,0:1:0\,]$, and these are in fact exactly the points of infinity valid for an arbitrary classical quadrature domain as in (\ref{dA}),
see \cite{Gustafsson-1988}. For the general hyperbola (\ref{hyperbola0}) the points of infinity become 
$[\,0:a-\I b:a+\I b\,]$ and $[\,0:a+\I b:a-\I b\,]$.  
\end{remark}

Solving (\ref{parabola}) for $\bar{z}$ gives the Schwarz function
\begin{equation}\label{Sparabola}
S(z)=z-4a+\sqrt{16a(a-z)},
\end{equation}
which has one branch point at $z=a$, the other being $z=\infty$. Let $D$ denote that component of
$\C\setminus\Gamma$ which contains the branch cut $[a,+\infty)$, and let $\Omega$ the complementary component
in $\C$.

It is known \cite{Sakai-1981} that $\Omega$ is
a null quadrature domain for planar area measure. The corresponding result is also true in higher dimensions
and can then be proved by exhausting the paraboloid by ellipsoids and using the main result
in \cite{Friedman-Sakai-1986}. This technique of exhaustions has been developed by H.~Shahgholian 
\cite{Shahgholian-1992}, L.~Karp  \cite{Karp-2008} and others.

For the spherical measure, $\Omega$ becomes a two point quadrature domain. The real parabola $\Gamma$ can be parametrized by
$z=x+\I y =f(\zeta)$, $ \zeta\in\R$, where 
$$
f(\zeta)=4a(\zeta^2 +\I \zeta).
$$
This function is meromorphic on the Riemann sphere and is univalent in the upper half plane $\U$, with $f(\U)=\Omega$.
Taking $M=\C\cup\{\infty\}$, $M_+=\U$, we are thus in a similar situation as for the hyperbola. 
The holomorphically reflected function with respect to $J(\zeta)=\bar{\zeta}$ is
$$
f^*(\zeta)=4a(\zeta^2-\I \zeta)
$$
and the quadrature identity becomes, as in (\ref{sfqi2}),
\begin{equation}\label{qiparabola}
\frac{1}{2\pi \I} \int_{\U} h\,\frac{df\wedge d\bar{f}}{(1+|f|)^2} 
=-\sum_{1+ff^*=0} {\rm Res\,} \Big(\frac{h  }{1+ff^*}\cdot\frac{df}{f}\Big).
\end{equation}

The equation $1+f f^*=0$ giving the quadrature nodes spells out to be $1+16a^2\zeta^2 (\zeta^2+1)=0$, with roots
\begin{equation}\label{zetaj}
\zeta=\zeta_j=\pm \sqrt{-\frac{1}{2}\pm \frac{1}{2}\sqrt{1-\frac{1}{4a^2}}}.
\end{equation}
These numbers  $\zeta_j$ ($j=1,2,3,4$) are non-real and come in complex conjugate pairs. 
The quadrature nodes are those two,  say $\zeta_1$ and $\zeta_2$, which are in the upper half
plane $\U$.  For the further analysis there are three cases: 
\begin{itemize}
\item $0<a<\frac{1}{2}$: There is one root in each quadrant. 
\item $a=\frac{1}{2}$: There are two double roots on the imaginary axis, $\zeta_1=\zeta_2=\frac{\I}{\sqrt{2}}$ and $\zeta_3=\zeta_4=-\frac{\I}{\sqrt{2}}$.
\item $\frac{1}{2}<a<\infty$: The double roots have been split into conjugate pairs of simple roots located on the imaginary axis. 
\end{itemize}

In the case of simple roots $\zeta_j\in \U$ we have
$$
-\sum_{1+ff^*=0} {\rm Res\,} \Big(\frac{h  }{1+ff^*}\cdot\frac{df}{f}\Big)= -\sum_{1+ff^*=0} \frac{h}{(1+ff^*)^\prime}\cdot \frac{f^\prime}{f}
$$
\begin{equation}\label{qiparabola1}
=-\frac{1}{32a^2}\sum_{j=1}^2 \frac{2\zeta_j+\I}{{\zeta_j^2(2\zeta_j^2+1)(\zeta_j+\I) }}\cdot h(\zeta_j),
\end{equation}
which in view of (\ref{zetaj}) makes the quadrature formula (\ref{qiparabola}) explicit. 
In the $z$-plane one gets 
$$
z_j=f(\zeta_j)={-2a\pm \sqrt{4a^2-1}}\pm \I\sqrt{-8a^2\pm 4a\sqrt{4a^2-1}},
$$
but simplifications of this expression seem difficult in general. Therefore we shall be more explicit 
only in the special case of double roots, namely when $a=\frac{1}{2}$. Then
$$
z_{1,2}=f(\zeta_{1,2})=f(\frac{\I}{\sqrt{2}})=-1-\sqrt{2},
$$
and the quadrature formula for $\Omega$ becomes, by straight-forward computations using (\ref{intres1}), (\ref{Sparabola}) directly,
$$
\frac{1}{2\pi \I }\int_\Omega h(z)\frac{d\bar{z}\wedge dz}{(1+|z|^2)^2}
=(1-\frac{1}{2\sqrt{2}})\cdot h(-1-\sqrt{2})+\frac{1+\sqrt{2}}{2}\cdot h'(-1-\sqrt{2}).
$$
Note that the coefficient $1-1/(2\sqrt{2})$ is slightly smaller than one ($=$ the area of the full sphere), as one might expect from  the exterior $\Omega$
of the parabola occupying a major part of the sphere. See Figure~\ref{fig:parabola}.


\begin{figure}
\begin{center}
\includegraphics[width=\textwidth, scale=1.4, trim=70 220 50 250]{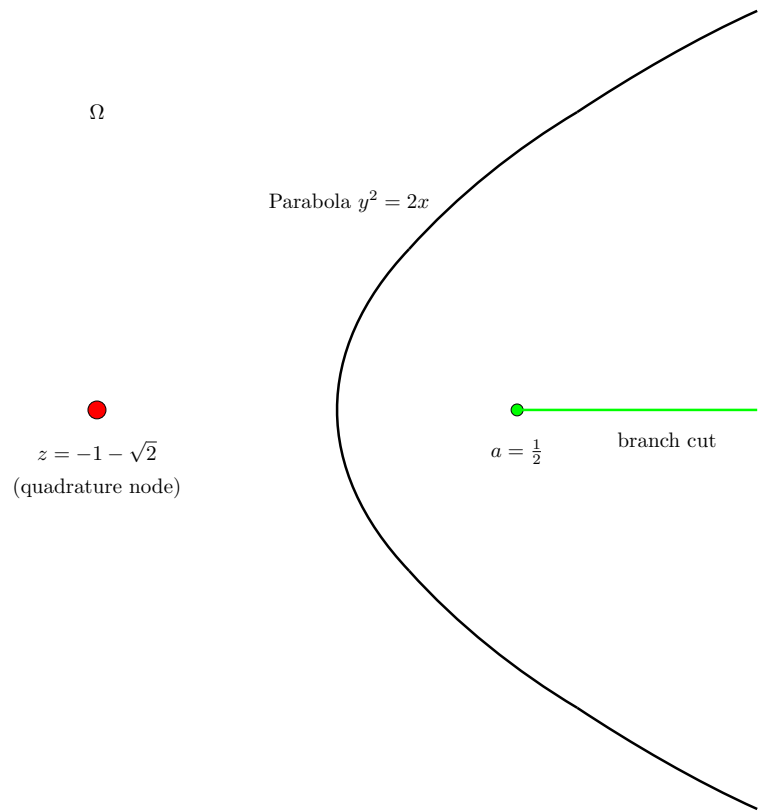}
\end{center}
\caption{The parabola with $a=\frac{1}{2}$. Quadrature node for exterior region, and of order two, in red. 
Branch point and branch cut (optional) in green. Interpretation similar to that for ellipse in Figure~\ref{fig:ellipse}.}
\label{fig:parabola}
\end{figure}


\section{Domains obtained by inversions of quadrics}\label{sec:inversions}

\subsection{The hippopede, or Neumann oval}

Inversion, for example by the antipodal map $z\mapsto -1/\bar{z}$, of the ellipse on the form (\ref{ellipse1}) gives
$$
z^2+\bar{z}^2-\frac{2(a^2+b^2)}{a^2-b^2}z\bar{z}+\frac{4a^2b^2}{a^2-b^2}z^2\bar{z}^2=0,
$$
or, on real form,
\begin{equation}\label{hippopede2} 
a^2b^2(x^2+y^2)^2-a^2y^2-b^2x^2=0.
\end{equation} 
See Figure~\ref{fig:hippopede} for the shape. 
It can be viewed as the  ``smash sum'' of two partially overlapping disks, 
compare \cite{Diaconis-Fulton-1990, Levine-peres-2010} for the planar case. 

We see that the equation now has degree four, and it describes a classical two point
quadrature domain, known as the Neumann oval or, according to some sources, hippopede, 
which is a special case of a hypocycloid. In the planar case (Euclidean metric), the quadrature identity for the domain 
$\Omega$ inside (\ref{hippopede2}) is quite simple:
$$
\frac{1}{2\pi\I}\int_\Omega h(z)d\bar{z}dz=\frac{a^2+b^2}{4a^2b^2}\Big( h(-\frac{\I c}{2ab})+h(\frac{\I c}{2ab})\Big).
$$
For the spherical metric the formula will be as in (\ref{qiellipsephysical}) along with (\ref{Aint}), but with the nodes $\pm z_{\rm int}$ moved by the
map which performs the inversion. In other words, the quadrature nodes will be $\pm 1/\bar{z}_{\rm int}$ with $z_{\rm int}$ as in (\ref{zint}).

The curve (\ref{hippopede2}) has degree $d=4$, but is still uniformized by the Riemann sphere, hence it has genus zero.
Therefore the genus formula (\ref{genus formula}) says that there must three singular points. Some of the singular points
come form the quadrature nodes: in the general setting of  (\ref{qi0}) the extension of the 
algebraic curve $\partial\Omega$ to $\C\P^2$ has, at each of the two points of infinity, $[\,0:1:0\,]$ and $[\,0:0:1\,]$, $m$ simple cusps of multiplicities
$n_1,\dots,n_m$ with distinct tangent directions. As the detailed analysis in \cite{Gustafsson-1988}
shows, this gives a total contribution $n(n-1)$ to the term for singular points in the genus formula (\ref{genus formula}). 
Here $n=\sum_{k=1}^m m_k$ is the order of the quadrature identity.

In the present case with the hippopede, $n=2$ and so $n(n-1)=2$. Thus there is one more singular point to account for.
This is what is called a ``special point'' \cite{Shapiro-1987, Shapiro-1992}, 
that is an isolated solution of $S(z)=\bar{z}$.
It is located at $[\,1:0:0\,] \in\C\P^2$, in other words at the origin in the complex plane. 


\begin{figure}
\begin{center}
\includegraphics[width=\textwidth, scale=1.8, trim= 5 450 5 200]{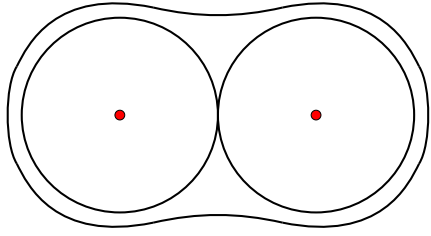}
\end{center}
\caption{Hippopede, along with limiting case consisting of two tangent disks. Quadrature node for Euclidean measure in red.
Those for spherical measure are located closer to the origin. (Compared to the equation (\ref{hippopede2}) in the text, the picture is 
rotated by 90 degrees.)}
\label{fig:hippopede}
\end{figure}


\subsection{The lemniscate}\label{sec:lemniscate}
 
The antipodal version of the hyperbola is a lemniscate. Indeed, the
inversion $z\mapsto -1/\bar{z}$ makes the equation (\ref{hyperbola}) transform into
\begin{equation}\label{hippopede}
z^2+\bar{z}^2-\frac{2(a^2-b^2)}{a^2+b^2}z\bar{z}-\frac{4a^2b^2}{a^2+b^2}z^2\bar{z}^2=0,
\end{equation}
or 
\begin{equation}\label{hippopede1} 
a^2b^2(x^2+y^2)^2+a^2y^2-b^2x^2=0,
\end{equation} 
a Bernoulli lemniscates. Since inversions are rigid transformations of the 
Riemann sphere the lemniscate is a quadrature domain in the same sense as the hyperbola. 
The curve has degree $d=4$ and has an obvious singularity at the origin (intersection between two smooth branches). 

The special case is $a=b=1$ is conspicuous. The lemniscate then becomes $z^2+\bar{z}^2=2z^2\bar{z}^2$, or 
\begin{equation}\label{special lemniscate}
|z^2-\frac{1}{2}|=\frac{1}{2},
\end{equation}
which is the inversion of the hyperbola $x^2-y^2=1$ in Figure~\ref{fig:hyperbola3}.
The quadrature identity takes the form (\ref{qihhprime}) with the node located on the curve, see Figure~\ref{fig:lemniscate}.

While hippopedes and cardioids (next section) are quadrature domains in the traditional sense,
lemniscate domains really require the Riemann surface setting to qualify.
In the same vein, lemniscates do not behave well with respect
to planar Laplacian growth processes, see \cite{Khavinson-Mineev-Putinar-Teororescu-2010}, one would need 
to consider such processes on covering surfaces, like in \cite{Gustafsson-Lin-2021}, in order for lemniscates to fit in.


\begin{figure}
\begin{center}
\includegraphics[width=\textwidth, scale=1, trim=5 400 5 200]{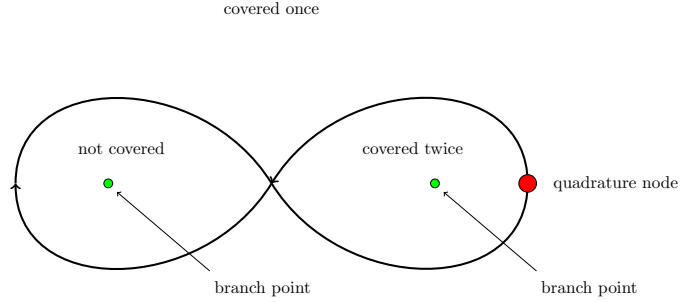}
\end{center}
\caption{The lemniscate (\ref{special lemniscate}) as a multi-sheeted quadrature domain.}
\label{fig:lemniscate} 
\end{figure}


\subsection{The cardioid}\label{sec:cardioid}

On antipodal inversion of the  parabola (\ref{parabola}) one gets an unbounded cardioid.
To keep the cardioid bounded one needs to invert with respect to a point not located on the
parabola itself. The best choice is the focal point $z=a$, and thus setting $w=(z-a)^{-1}$
one gets in the $w$-plane a bounded cardioid having a cusp at $w=0$. 
It can be parametrized from the unit disk by 
$$
w=\frac{1}{2a}\big(\zeta-\frac{1}{2}\zeta^2\big)-\frac{1}{4a} \quad (\zeta \in\D),
$$
and it is a well-studied quadrature domain. In the planar case (Euclidean area measure)
there is a quadrature node of order two at $w=-1/4a$. 

The equation for the boundary becomes, with $w=u+\I v$,
$$
4a^2(u^2+v^2)^2+4au(u^2+v^2)-v^2=0.
$$
again of degree $d=4$. In similarity with the case of the hippopede, the double quadrature node 
accounts for two of the three singular points dictated by the genus formula. The remaining singular point
is the cusp at the origin.


\section{Appendix: Some computational details for the proof of Theorem~\ref{thm:qi}}\label{sec:appendix}

To prove the final statements of Theorem~\ref{thm:qi}, namely to confirm the presence of cancellations leading to (\ref{sfqi2}), 
we shall make a detailed local analysis of the differentials
\begin{equation}\label{differentials}
\frac{hdf}{f(1+ff^*)}, \quad \frac{hf^*df}{1+ff^*}
\end{equation}
near an arbitrary point, which we take to be the origin in a local coordinate $\zeta$.
We assume the following expansions as $\zeta\to 0$, where $a,b,c\ne 0$ and $k,\ell,j$ are integers: 
$$
f(\zeta)= a\zeta^k+\text{higher powers},
$$
$$
f^\prime(\zeta)=b\zeta^\ell+\text{higher powers},
$$
$$ 
f^*(\zeta)=c\zeta^j+\text{higher powers}.
$$
If $k=0$ then $\ell\geq 0$. If $k\ne 0$ then $\ell=k-1$, $b=ka$. For the test functions we normalize the leading coefficient to be one: 
$$
h(\zeta)=\zeta^r +\text{higher powers} \quad (r\geq 0).
$$

We shall prove that under the stated assumptions in Theorem~\ref{thm:qi} there are cancellations in formula (\ref{sfqi1}), repeated here:
$$
\frac{1}{2\pi \I} \int_{M_+} h\,\frac{d\bar{f}\wedge d{f}}{(1+|f|^2)^2} 
=-\sum_{M_+} {\rm Res\,} \frac{h  df}{f(1+ff^*)} +\sum_{\omega_k\in M_+} h(\omega_k),
$$
such that the right member takes the simpler form 
$$
-\sum_{1+ff^*=0} {\rm Res\,} \Big(\frac{h  }{1+ff^*}\cdot\frac{df}{f}\Big)
$$
appearing in (\ref{sfqi2}).
For this purpose we need to identify the residue at $\zeta=0$ of the first differential in (\ref{differentials}), which is
\begin{equation}\label{residuedifferential1}
\frac{hdf}{f(1+ff^*)}=\frac{b}{a}\cdot\frac{\zeta^{r+\ell-k} d\zeta+\dots}{1+ac\zeta^{k+j}+\dots}.
\end{equation}
There are three main cases, and some subcases of these.

\begin{itemize}
\item
$k+j>0$:  In this case there is a residue if and only if $r+\ell-k=-1$. 
\begin{itemize}
\item If $k=0$ then $\ell\geq 0$, hence $r+\ell -k\geq 0$, and there is {\bf no residue}.
\item If $k> 0$ then $r+\ell -k=r-1$ and there is a residue if $r=0$.  The size of the residue is then $b/a=k$. 
However, this {\bf residue cancels} with $k$ terms named $h(\omega_k)$ (the latter ``$k$'' being just an index)
in the first formula in the theorem.
\item If $k<0$ then $j>0$. This is {\bf excluded by assumption}.
\end{itemize}

\item $k+j=0$: In this case the denominator in (\ref{residuedifferential1}) is $1+ac +\mathcal{O}(\zeta)$. 

\begin{itemize}
\item If $1+ac\ne 0$, then the $\mathcal{O}(\zeta)$-term is not the main term, and we again get a residue contribution
if an only if $r+\ell-k=-1$. The rest of the analysis is exactly the same as in the case  $k+j>0$.

\item
If $1+ac=0$, then $1+ff^*$ has a zero at $\zeta=0$, of order $s\geq 1$ say: 
$$
1+f(\zeta)f^*(\zeta)=q\zeta^s+ \mathcal{O}(\zeta^{s+1}),\quad q\ne 0.
$$ 
This means that the differential in (\ref{residuedifferential1}) looks like
$$
\frac{hdf}{f(1+ff^*)}=\frac{b}{aq}\cdot \zeta^{r+\ell-k-s}d\zeta+\text{higher powers},
$$ 
and there is a residue contribution if $r=s+k-\ell -1$.

This case represents the {\bf surviving residues} in equation (\ref{sfqi1}), namely those coming from points where $1+ff^*=0$.

\end{itemize}

\item $k+j<0$: Then the leading power in (\ref{residuedifferential1}) is $\zeta^{r+\ell-2k-j}d\zeta$, and there is a residue contribution if and only if
$r+\ell-2k-j=-1$. 

\begin{itemize}
\item If $k=0$ then $j<0$, $\ell\geq  0$, hence $r+\ell-2k-j> 0$ and there is {\bf no residue}. 
\item If $k\ne 0$, then $r+\ell -2k-j=r-1-(k+j)\geq r\geq 0$, and there is {\bf no residue}. 
\end{itemize}

\end{itemize}

To prove (\ref{sfqi2}) we investigate  the second differential in (\ref{differentials}):
\begin{equation}\label{residuedifferential2}
\frac{hf^*df}{1+ff^*}={b}{c}\cdot\frac{\zeta^{r+j+\ell} d\zeta+\dots}{1+ac\zeta^{k+j}+\dots}.
\end{equation}
There are again several cases.

\begin{itemize}
\item $k+j>0$:  In this case there is a residue contribution at $\zeta=0$ if and only if $r+j+\ell=-1$. The size of the residue is $bc$.

\begin{itemize}
\item If $k=0$, then $j>0$, hence $r+j+\ell > 0$ and there is {\bf no residue}.
\item If $k\ne 0$, then $r+j+\ell=r+j+k-1\geq r\geq 0$, with still {\bf no residue}.
\end{itemize}

\item $k+j=0$ and $1+ac\ne 0$: Again there is residue provided $r+j+\ell =-1$.

\begin{itemize}
\item If $k=0$, hence $j=0$, then $r+j+\ell\geq 0$ and there is {\bf no residue}.
\item If $k\ne 0$, so that $r+j+\ell=r+j+k-1=r-1$, there is a residue if $r=0$.
The size of the residue is $bc=kac$.

\begin{itemize}
\item If $k>0$ then $j<0$. This is {\bf excluded by assumption}.

\item If $k<0$ then $j>0$. This is {\bf excluded by assumption}.

\end{itemize}

\end{itemize}

\item
$k+j=0$ and $1+ac=0$. Then $1+ff^*$ has a zero at $\zeta=0$, of order $s\geq 1$ say: 
$$
1+f(\zeta)f^*(\zeta)=q\zeta^s+ \mathcal{O}(\zeta^{s+1}),\quad q\ne 0.
$$ 
The differential in (\ref{residuedifferential2}) now is 
$$
\frac{hf^*df}{1+ff^*}=\frac{bc}{q}\cdot \zeta^{r+j+\ell-s}d\zeta+\text{higher powers}.
$$ 
and there is a residue if  $r=s+k-\ell-1$.

This case represents the {\bf surviving residues} in equation (\ref{sfqi2}), namely those coming from points where $1+ff^*=0$.

\item $k+j<0$: Then the leading power in (\ref{residuedifferential2}) is $\zeta^{r+\ell-k}d\zeta$, and there is a residue contribution if and only if
$r+\ell-k=-1$. 

\begin{itemize}
\item If $k=0$ then  $\ell\geq 0$, so $r+\ell -k\geq 0$ and there is {\bf no residue}. 
\item If $k\ne 0$, then $r+\ell -k=r-1$ and there is a residue if $r=0$. 

\begin{itemize}
\item If $k>0$ then $j<0$. This is {\bf excluded by assumption}.
\item If $k<0$ there is a contribution, by $bc/ac=b/a=k=-|k|$. This {\bf residue cancels} together with $|k|$ terms 
named $h(\zeta_j)$ in the second formula of the theorem. 
\end{itemize}
\end{itemize}

\end{itemize}

The above analysis confirms the quadrature identities (\ref{sfqi2}) under the stated assumptions.
In fact, it shows that even without such assumptions (of non-coinciding poles and zeros) the terms with explicit zeros 
and poles in the first formas (\ref{sfqi1}) in the theorem are cancelled by terms in the residue sums. But in general it becomes 
difficult to specify the ranges of these sums.



\bibliography{bibliography_gbjorn.bib}

\end{document}